\documentclass[12pt]{article}
\usepackage[latin1]{inputenc}
\usepackage{amsmath,amsthm,amssymb}

\usepackage[all]{xy}

\usepackage[colorlinks=true,linkcolor=red,citecolor=red]{hyperref}

\usepackage{mathrsfs}

\newcommand{\Aut}{\textrm{Aut}}

\newcommand{\bs}[1]{\boldsymbol{#1}}

\newcommand{\Ext}{\mathrm{Ext}}
\newcommand{\Exp}{\mathrm{Exp}}

\newcommand{\V}{\mathrm{V}}

\newcommand{\Hom}{\mathrm{Hom}}
\newcommand{\Mod}{\mathrm{Mod}}

\newcommand{\Reg}{\mathrm{Reg}}

\renewcommand{\H}{\mathrm{H}}

\newcommand{\M}{\mathrm{M}}

\newcommand{\lr}[1]{\langle#1\rangle}
\newcommand{\f}[1]{(\!(#1)\!)}

\newcommand{\Mon}{\mathrm{\mathbf{Mon}}}

\newcommand{\Rep}{\mathrm{Rep}}

\newcommand{\Rm}{\mathrm{\mathbf{R}}}

\newcommand{\Sol}{\textrm{Sol}}

\newcommand{\sm}[1]{\begin{smallmatrix}#1\end{smallmatrix}}

\newcommand{\simto}{\xrightarrow{\;\sim\;}}

\newcommand{\comm}[1]{
{\noindent\begin{color}{red} \framebox{\framebox{\framebox{
\begin{minipage}{450pt}#1
\end{minipage}}}}\end{color}}
}

\usepackage{amsthm}
\swapnumbers

\makeatletter
\def\swappedhead#1#2#3{%
  \thmname{#1}\;%
  \thmnumber{\@upn{\the\thm@headfont#2\@ifnotempty{#1}}}%
  \thmnote{\,{\the\thm@notefont(#3)}}{.~}}
\makeatother

\newtheoremstyle{dotless-thm}
  {10pt}
  {10pt}
  {\itshape}
  {}
  {\bfseries}
  {}
  {.0em}
  {}
\theoremstyle{dotless-thm}

\newtheorem{theorem}{\textbf{Theorem}}[subsection]
\newtheorem{def-intro}{\textbf{\textsc{Definition}}}
\newtheorem{thm-intro}{\textbf{\textsc{Theorem}}}
\newtheorem{rk-intro}{\textbf{\textsc{Remark}}}
\newtheorem{cor-intro}[thm-intro]{\textbf{\textsc{Corollary}}}
\newtheorem{Ex-intro}{\textbf{\textsc{Example}}}
\newtheorem{proposition}[theorem]{\textbf{Proposition}}
\newtheorem{lemma}[theorem]{\textbf{Lemma}}
\newtheorem{corollary}[theorem]{\textbf{Corollary}}
\newtheorem{definition}[theorem]{\textbf{Definition}}
\newtheorem{remark}[theorem]{\textbf{Remark}}
\newtheorem{example}[theorem]{\textbf{Example}}
\newtheorem{hypothesis}[theorem]{\textbf{Hypothesis}}

\newtheorem{notation}[theorem]{\textbf{Notation}}

\newcounter{qstct}

\newcounter{numero-obj}
\newcounter{numero-er}

\usepackage{pdfpages}

\title{Regular singular differential modules over differential rings}
\author{Andrea Pulita}
\date{\today}

\begin{document}
\maketitle
\begin{abstract}
We obtain Fuchs decomposition theorem for regular singular differential 
modules over a large class of differential rings. We provide a definition of 
regularity inspired by differential Galois theory and we deduce the classical 
equivalence with vector spaces endowed with an automorphism.
\end{abstract}
\makeatletter
\renewcommand\tableofcontents{%
    \subsection*{\contentsname}%
    \@starttoc{toc}%
    }
\makeatother

\begin{small}
\setcounter{tocdepth}{1} \tableofcontents
\end{small}

\setcounter{section}{0}


\section*{\textsc{Introduction}}
\addcontentsline{toc}{section}{\textsc{Introduction}}
A linear homogeneous ordinary differential 
equation\footnote{The notation $y^{(n)}$ means the 
$n$-th derivative $(\frac{d}{dt})^n(y)$.}
\begin{equation}\label{eq-intro : equation}
y^{(n)}(t)+
f_{n-1}(t)y^{(n-1)}(t)+\cdots+f_1(t)y'(t)+f_0(t)y(t)\;=\;0
\end{equation}
with coefficients $f_i(t)$ in the field of Laurent power series 
$\mathbb{C}\f{t}$ over the complex numbers 
$\mathbb{C}$ may have no formal power series solutions $y(t)\in\mathbb{C}\f{t}$. The classical theory of exponents (see, for instance, \cite{Ince}), investigates the possibility of having a solution 
of this equation in the form
\begin{equation}\label{eq : sol t^a}
y\;=\;t^e(a_0+a_1t+a_2t^2+\cdots)
\end{equation}
where $e\in\mathbb{C}$ and $a_i\in\mathbb{C}$. 
Here, $ t^e $ serves as a symbolic representation that, in the framework of 
differential Galois theory, can be understood as an abstract element within a 
suitable differential field extension of $ \mathbb{C}\f{t}$ (cf. \cite{VS}).
In case such a solution exists, the number 
$e\in\mathbb{C}$ is called an 
\emph{exponent} of the equation. 
The differential equation is called \emph{regular singular}, or simply regular, if all its solutions 
are polynomial expressions of $\log(t)$ and of solutions of type 
\eqref{eq : sol t^a} (cf. \cite[Exercises 3.13 and 3.29]{VS}). 
The fact that \eqref{eq-intro : equation} 
is regular singular can be tested directly, by looking at 
the order of poles (i.e. the $t$-adic valuation) of the coefficients $f_i$ 
(cf. \cite[Definition 3.14 and Proposition 3.16]{VS}). 
This last characterization is indeed often taken as the 
definition of regular singular, while the fact of having solution of the 
prescribed form is considered as an (equivalent) consequence.

A more algebraic approach lead to consider 
\emph{differential modules}, i.e. modules over the Weyl algebra 
(cf. \cite{Levelt-Hypergeom, Deligne-reg-sing, Katz-Nilpotent, VS}). 
In this case, when the exponents exist, 
only their image in $\mathbb{C}/\mathbb{Z}$ is invariant 
by base changes in the differential module. Indeed, we may multiply 
\eqref{eq : sol t^a} by $t^n$ and obtain another solution of the same 
differential module (in another basis). 
Differential modules over $\mathbb{C}\f{t}$ are completely classified 
(cf. \cite{VS} and therein bibliography). 
In particular, every differential module over $\mathbb{C}\f{t}$ 
is direct sum of a regular differential module 
and a completely irregular module (i.e. a module without 
solutions of the form \eqref{eq : sol t^a}). 
Moreover, regular singular differential modules over $\mathbb{C}\f{t}$ 
admit the so called \emph{Fuchs decomposition}, 
that is, the differential module admits a Jordan-Hölder sequence of rank 
one regular singular differential modules that are associated (up to isomorphism) 
with a simple equation of the form
\begin{equation}\label{eq : y=t^e intro}
y'(t)=\frac{e}{t}y(t)\;,\quad y(t)=t^e\;.
\end{equation}
Exponents are crucial invariants of the differential module. 
For instance, they are useful to classify 
differential modules, to prove algebraicity of their solutions, or to compute 
 their monodromy representations. Indeed, the differential Galois group of the category of regular modules over 
$\mathbb{C}\f{t}$ is the algebraic envelop of $\mathbb{Z}$, 
whose generator is given by the formal monodromy operator 
(cf. \cite[Section 3.2]{VS}). In other words, the category of regular modules 
over 
$\mathbb{C}\f{t}$ is equivalent to that of $\mathbb{C}$-linear finite 
dimensional representations of $\mathbb{Z}$, 
which are nothing but finite dimensional $\mathbb{C}$-vector spaces 
together with an automorphism (the monodromy operator). In this 
correspondence, Fuchs decomposition theorem corresponds to Jordan 
decomposition of endomorphisms by eigenspaces, and the multiset of 
exponents of a regular differential module coincide, in a suitable sense, 
with the logarithm of the multiset of eigenvalues of the automorphism of 
the corresponding representation of $\mathbb{Z}$.

The problem of finding solutions of the form \eqref{eq : sol t^a}
has a interest over differential rings 
that are more general than the field $\mathbb{C}\f{t}$. 
For instance, similar results exist for differential modules 
over fields of Laurent power series that are convergent on 
some unspecified complex or $p$-adic disk with zero removed (cf. 
\cite{Clark,Clark-2, VDP-liouville, Ba, VDP-liouville-2}) 
or over the ring of analytic functions over a $p$-adic annulus 
(cf. \cite{Ch-Me-II, Dwork-exponents, Kedlaya-draft, Kedlaya-draft-bis}).

In this last context, G. Christol and Z. Mebkhout established the $p$-adic analog of the Fuchs decomposition theorem for a specific class of differential equations satisfying the Robba condition on a $p$-adic annulus, under additional technical assumptions on the exponents. In particular, the differences between the exponents are required to be $p$-adic non-Liouville numbers (see \cite{Ch-Me-II}). 

The Robba condition requires the Taylor solutions of the differential equation to converge maximally at every point of the annulus, including those in extensions of the base field. Differential modules satisfying the Robba condition are, \emph{a priori}, not necessarily cyclic\footnote{To the best of our knowledge, no explicit examples of non-cyclic modules have been constructed.}. This means they may not correspond to a differential equation of the form \eqref{eq-intro : equation}. 

Even when such modules are cyclic, the Robba condition, as well as the constraints on the $p$-adic exponents, cannot (at present) be expressed in terms of the valuations of the coefficients $f_i$. For example, consider the simple differential equation \eqref{eq : y=t^e intro}. A change of basis in the associated differential module may transform it into an equation of the form \( y' = g(t)y \), where \( g(t) \) has arbitrarily large poles at \( t = 0 \). 

A similar phenomenon occurs under changes of the coordinate \( t \) in the $p$-adic annulus, further complicating the situation.

Therefore, it does not seem reasonable (up to today) to expect 
to describe the exponents, or even only the fact of being 
``\emph{regular singular}'', in terms of the coefficients $f_i$ of 
\eqref{eq-intro : equation}. Nevertheless, the $p$-adic Fuchs theorem by Christol and Mebkhout demonstrates that the solutions are polynomial combinations of $\log(t)$ and expressions of the form $t^e \cdot f(t)$, where $f(t)$ is an analytic function defined over the $p$-adic annulus, and $e \in \mathbb{Z}_p$ is a $p$-adic integer.\footnote{Thanks to the Fuchs decomposition, one can further establish \emph{a fortiori} that these modules are, in fact, cyclic.}
However, the class of differential modules considered by Christol and 
Mebkhout is somewhat restrictive. In particular, the Robba condition 
excludes simple examples of differential equations, such as 
\eqref{eq : y=t^e intro}, where $e \notin \mathbb{Z}_p$. While the Robba 
condition plays a crucial role in the technical framework of their proof of 
the Christol-Mebkhout $p$-adic Fuchs theorem, it may not be the most 
practical criterion for defining regular singular differential equations. 
Indeed, as there exists explicit examples of rank 2 differential modules 
satisfying the Robba property which have no non trivial sub-quotients (cf. 
\cite{Ch-exponent}), it follows that the Tannakian group 
of the category of differential equations of type Robba is 
strictly larger than the algebraic envelop of $\mathbb{Z}$.

The following problem naturally arises: \emph{Let $K$ be an 
algebraically closed field of characteristic zero and $A$ a 
$K$-algebra with derivation containing the ring of Laurent polynomials 
$K[t,t^{-1}]$. 
What is a convenient definition of regular singular differential module over 
$A$ capturing the largest class of differential modules ? Can we identify 
minimal assumptions on $A$ that allow for the definition of exponents and 
ensure a Fuchs decomposition?}

To address this question, inspired by differential Galois theory over $\mathbb{C}\f{t}$, we define regular singular differential modules over a differential ring $A$ as those whose abstract solutions involve polynomials in $\log(t)$ and terms of the form $t^e \cdot f$, where $f \in A$ (cf. Definition \ref{Def : reg sing diff mod}). Remarkably, we show that minimal assumptions on $A$ suffice to ensure the Fuchs decomposition theorem. Specifically, it is enough that the only solutions in $A$ of $\partial^2(y) = 0$ are constants and that $\partial(y) + ay = 0$ has no nonzero solutions in $A$ for any $a \in K \setminus \mathbb{Z}$ (cf. Definition \ref{Def : diff alg no exp nor log}), ensuring that $\log(t)$ and $t^a$ do not already belong to $A$. Under these assumptions, we establish an equivalence 
between the category of regular singular differential modules over $A$ 
and the category of 
$K$-linear representations of $\mathbb{Z}$, i.e., $K$-vector spaces 
equipped with a $K$-linear automorphism (cf. Theorem 
\ref{Thm. Equivalence}). Consequently, the Tannakian 
group associated with the category of regular singular modules is 
identified as the algebraic envelope of $\mathbb{Z}$, mirroring the 
classical case. In particular, every regular singular module admits a 
Jordan-Hölder sequence made by rank one regular singular sub-quotients.

\if{
\begin{thm-intro}[\protect{cf. Theorem \ref{}}]
\end{thm-intro}

\comm{Migliorare}

...

...

....

Completare ...

...

...
}\fi

\section{General Notation}\label{Section : general notation}
We denote by $\mathbb{Z}=\{0,1,-1,2,-2,\ldots\}$ the 
ring of integer numbers and by $\mathbb{Q}$ its fraction 
field, the field or rational numbers. We denote by 
$\mathbb{R}$ the field of real numbers\footnote{The 
completion of $\mathbb{Q}$ with respect to the 
archimedean absolute value.} and by 
$\mathbb{C}$ that of complex numbers\footnote{The 
algebraic closure of $\mathbb{R}$.}.

If $R$ is a ring, we denote by $R^\times:=\{r\in R\;\textrm{ such that }\;\exists s\in R, sr=1\}$ the group of 
invertible elements of $R$.

In this paper $K$ will be an algebraically closed field of 
characteristic $0$. We denote by $(K,+)$, or simply $K$, 
its additive group and by $K^\times=K-\{0\}$ its 
multiplicative group. 

We denote by $\mu_\infty\subset K^\times$ the 
subgroup of all roots of unity in $K$.

\begin{hypothesis}\label{ZFC}
We assume in the paper the axiom of choice (ZFC) 
in Zermelo-Fraenkel set theory, 
so that every infinite cardinal is 
the infinite cardinals of well-ordered sets $\omega$ and 
this transfinite sequence is totally ordered.
\end{hypothesis}
If $L:V\to V$ is a $K$-linear endomorphism of vector 
space we denote by
\begin{eqnarray}
V^{L=0}&\;:=\;&\mathrm{Ker}(L:V\to V)\;\\
V^{L=1}&\;:=\;&\mathrm{Ker}(L-\textrm{id}:V\to V).
\end{eqnarray}
Let $R$ be a ring.
If $I$ is a set of indexes and $X=\{X_i\}_{i\in I}$ is a 
family of variables, we denote by $R[\{X_i\}_{i\in I}]$ 
the ring of polynomials in the variables 
$\{X_i\}_{i\in I}$ (cf. 
\cite[Chapitre III, \S 2, N.9]{Bou-Alg-II}). 

Let $(G,\cdot)$ be a group, written 
multiplicatively. We denote by 
$R\lr{G}$
the \emph{group ring of $G$ with coefficients in $R$}. 
Recall that $R\lr{G}$ is a free $R$-module with a 
basis given by a 
family of symbols 
$\{\lr{g}\}_{g\in G}$ indexed by $G$ 
(i.e. $R\lr{G}=\oplus_{g\in G}R\lr{g}$), 
and that the 
multiplication law is determined by the relation 
$\lr{g}\lr{h}=\lr{gh}$ extended by linearity:
\begin{equation}
(\sum_{g\in G}a_g\lr{g})\cdot(\sum_{g\in G}b_g\lr{g})\;=\;\sum_{g\in G}(\sum_{h\in G}a_hb_{h^{-1}g})\lr{g}\;,\quad \quad a_g,b_g\in R\;. 
\end{equation}
If $G$ is commutative, we will also use another notation which will be 
more evocative for our purposes
\begin{equation}
t^g\;=\;\lr{g}\;
\end{equation}
so that every element $x\in R\lr{G}$ can be uniquely 
written as $\sum_{g\in G}r_gt^g$, with $r_g\in R$. 
In particular, if $G=\mathbb{Z}$, then
the group ring $R\lr{\mathbb{Z}}$ is naturally 
isomorphic to the ring $R[t,t^{-1}]$, that is, the 
localized of the polynomial ring
$R[t]$, in one variable, along the multiplicative subset 
$S=\{t^n\}_{n\geq 0}$.
We denote by 
\begin{eqnarray}
\partial_t\;=\;t\frac{d}{dt}
\end{eqnarray}
the derivation on $R[t,t^{-1}]$ given by 
\begin{equation}\label{eq : def partialt}
\partial_t\Bigl(\sum_{i\in\mathbb{Z}}r_it^i\Bigr)\;=\;
\sum_{i\in\mathbb{Z}}r_i\cdot i\cdot t^i\;.
\end{equation}
We also write $\partial=\partial_t$ if no confusion is 
possible.

In the next sections we define \emph{regular singular 
differential modules} over a differential ring $A$, 
satisfying some minimal properties, and 
we prove that regular singular modules over $A$ 
form a Tannakian category whose
group is the algebraic envelop of $\mathbb{Z}$. 

\section{The ring of exponents}
We consider the group algebra $K\lr{K}$ of the additive 
group $(K,+)$. 
\if{
It is a vector 
space with basis $\{t^a\}_{a\in K}$, where $t^a$ is 
merely a symbol, so that every 
element in $K\lr{K}$ can be uniquely written as a finite sum 
$\sum_{a\in K}b_at^a$ with $b_a\in K$. It is a ring 
under the rule $t^{a_1}t^{a_2}=t^{a_1+a_2}$. 
}\fi
In order to emphasize our link with the theory of 
differential equations for every set $I$ we set 
$t^I=\{t^a\}_{a\in I}$ and we introduce the notation
\begin{eqnarray}
K\lr{K}\;=\;K[\{t^a\}_{a\in K}]\;=\;K[t^K]\;.
\end{eqnarray}
The inclusion $(\mathbb{Z},+)\subset (K,+)$ provides an 
inclusion
\begin{eqnarray}\label{eq : KZ-KK}
K[t,t^{-1}]\;\subseteq\;K[t^K]\;.
\end{eqnarray} 
where 
$K[t,t^{-1}]=K\lr{\mathbb{Z}}=K[t^{\mathbb{Z}}]$ 
is the classical ring of Laurent polynomials with coefficients in $K$. 
Now, let us 
chose a set theoretical section of the projection $K\to K/\mathbb{Z}$ and denote its image by
\begin{eqnarray}\label{eq : widetilde K/Z}
\widetilde{K/\mathbb{Z}}\subseteq K\;.
\end{eqnarray}
It is a set of representatives in $K$ of the elements of 
$K/\mathbb{Z}$. By convenience of notation 
we chose $0\in \widetilde{K/\mathbb{Z}}$ as the 
lifting of $0\in K/\mathbb{Z}$. 
Then $K[t^K]$ is a free 
$K[t,t^{-1}]$-module with basis 
$t^{\widetilde{K/\mathbb{Z}}}$. 
In particular, every element $x\in 
K[t^K]$ can be uniquely written as 
\begin{eqnarray}\label{uniquenedd K}
x\;=\;\sum_{a\in \widetilde{K/\mathbb{Z}}}f_a(t) 
t^a\;,\qquad f_a(t)\in K[t,t^{-1}]\;.
\end{eqnarray}
Let now $A$ be a commutative ring with unit element 
containing $K[t,t^{-1}]$ as a subring. 
De note the inclusion by
\begin{equation}\label{eq: rho}
\rho\;:\;K[t,t^{-1}]\xrightarrow{\quad} A\;.
\end{equation}
Recall that a \emph{differential ring} is a pair $(A,\partial)$ where 
$A$ is a commutative ring with unit element and 
$\partial:A\to A$ is a derivation. 
\begin{hypothesis}\label{Hyp : A is a diff ring}
We assume from now on that $A$ is a differential ring whose derivation 
$\partial:A\to A$ extends $t\frac{d}{dt}:K[t,t^{-1}]\to K[t,t^{-1}]$ 
(cf. \eqref{eq : def partialt}).\footnote{That is $\partial(f)=t\frac{d}{dt}(f)$, for all $f\in K[t,t^{-1}]$.}
\end{hypothesis}
We set
\begin{eqnarray}
A[t^K]\;:=\;A\otimes_{K[t,t^{-1}]}K[t^K]\;.
\end{eqnarray}
It is a free $A$-module with basis 
$t^{\widetilde{K/\mathbb{Z}}}$, so every element $x\in 
A[t^K]$ can be uniquely written as 
\begin{eqnarray}\label{x written as At^K}
x\;=\;\sum_{a\in \widetilde{K/\mathbb{Z}}}f_a 
\cdot t^a\;,\qquad f_a\in A\;.
\end{eqnarray}
Let $\ell$ be a new indeterminate. 
We may form the ring of polynomials in $\ell$ 
with coefficients in the ring $A[t^K]$.
\begin{eqnarray}
E_A\;:=\;A[t^K][\ell]\;.
\end{eqnarray}
It is a free $A$-module with basis $t^{\widetilde{K/\mathbb{Z}}}\cup \{\ell^n\}_{n\geq 1}$ and every element $x\in E_A$ can be uniquely written 
as
\begin{equation}\label{eq: x written in E_A}
x\;=\;\sum_{k=0}^n\Bigl(\sum_{a\in\widetilde{K/\mathbb{Z}}}f_{a,k}\cdot t^{a}\Bigr)\ell^k\;=\;\sum_{a\in\widetilde{K/\mathbb{Z}},\;\; k=0,\ldots,n}f_{a,k}\cdot t^{a}\cdot\ell^k\;, \quad\qquad f_{a,k}\in A\;.
\end{equation}
\begin{definition}
We call $E_A$ the ring of exponents with respect to 
$A$.
\end{definition}

\if{

It may be convenient to write the value of $\partial$ on 
an element  $x\in E_A$ written as in 
\eqref{eq: x written in E_A}
\begin{equation}
\partial(x)\;=\;
\Bigl(\sum_{k=0}^{n-1}
\Bigl(\Bigl(\sum_{a\in\widetilde{K/\mathbb{Z}}}(\partial(f_{a,k})+a\cdot f_{a,k})t^a\Bigr) - 
(k+1)\Bigl(\sum_{a\in\widetilde{K/\mathbb{Z}}}f_{a,k+1}\cdot t^a\Bigr)\Bigr)\ell^k\Bigr)+\Bigl(\sum_{a\in\widetilde{K/\mathbb{Z}}}(\partial(f_{a,k})+a\cdot f_{a,k})t^a\Bigr)\ell^n
\end{equation}

}\fi
The following Lemma shows that $\ell$ plays the role of 
the usual functions $\log(t)$ in classical analysis and it justifies the notation 
$t^a$ for the elements of the group algebra.
\begin{lemma}\label{Lemma : Existence of partial}
The ring $E_A$ has a unique derivation, still denoted by
\begin{eqnarray}
\partial\;:\;E_A\to E_A,
\end{eqnarray}
extending $\partial:A\to A$ such that
\begin{enumerate}
\item for all $a\in K$, $a\neq 0$ one has 
$\partial(t^a)=at^a$; 
\item $\partial(\ell)=1$.
\end{enumerate}
Moreover $\partial$ stabilizes globally the subring $A[t^K]$  
and it induces a derivation on it still denote by $\partial$. 
\end{lemma}
\begin{proof}
For every $x\in A[t^K]$ written as in 
\eqref{x written as At^K} we must have
\begin{eqnarray}\label{eq : x=sum_ainK/Z (partial(f_a)+af_s)t^a}
\partial(x)\;=\;\sum_{a\in\widetilde{K/\mathbb{Z}}}
(\partial(f_a)+af_a)\cdot t^a
\end{eqnarray}
The uniqueness of the coefficients in 
\eqref{x written as At^K} implies that this provides a 
well defined derivation on $A[t^K]$. It is also clear that 
it extends uniquely to a derivation of $E_A=A[T^K]\otimes_KK[\ell]$.
\end{proof}
Let us introduce a terminology for the class of rings that 
behave well for our purposes. Recall Hypothesis \ref{Hyp : A is a diff ring}.
\begin{definition}\label{Def : diff alg no exp nor log}
\label{Def : algebras without exp nor log}
We say that $A$ is a \emph{differential $K$-algebra 
without exponents nor logarithm} if 
the following conditions hold
\begin{enumerate}
\item\label{Def : diff alg no exp nor log-3} One has
\begin{equation}\label{eq : A^d^2=K}
A^{\partial^2=0}\;=\;K\;.
\end{equation}
\item\label{Def : diff alg no exp nor log-4} 
For all $a\in\widetilde{K/\mathbb{Z}}-\{0\}$, one 
has 
\begin{equation}\label{eq : A^partial+a=0}
A^{\partial+a\cdot \mathrm{id}=0}\;=\;0\;.
\end{equation}
\end{enumerate}
\end{definition}
\begin{remark}
For $\lambda,\mu\in K$ and $a\in K-\mathbb{Z}$, 
the elements $\lambda t^a\in E_A$ and 
$\lambda\ell+\mu\in E_A$ are solutions of the differential 
equations $\partial(y)=a y$ and $\partial^2(y)=0$ 
respectively.  In other words they belongs to the kernels 
of the operators $\partial-a$ and $\partial^2$ acting on $A$.
In classical differential Galois theory (over a given differential field), the 
dimension of the space of solutions of a linear differential 
operator of order $n$ is bounded by $n$, hence the above differential 
equations would not have further solutions in $A$.
However, if the ring $(A,\partial)$ is general, these differential 
equations may actually have more solutions in $A$. 
\emph{Differential $K$-algebra 
without exponents nor logarithm} are precisely those for which one 
has the ``\emph{right}'' solutions for those two differential equations.
\end{remark}
\begin{proposition}\label{Prop : kernel of partial}
The differential ring $A$ is a \emph{differential $K$-algebra 
without exponents nor logarithm} if, and only if, we have
\begin{equation}\label{eq : At^Kpartial=E_A^partial=K}
A^{\partial=0}\;=\;A[t^{K}]^{\partial=0}\;=\;E_A^{\partial=0}\;=\;K\;.
\end{equation}
\end{proposition}
\begin{proof}
Let us assume that $A$ is a {differential $K$-algebra 
without exponents nor logarithm}. 
It is clear that \eqref{eq : A^d^2=K} implies $A^{\partial=0}=K$.

If $x\in A[t^K]$, we may write it as in 
\eqref{x written as At^K} and use \eqref{eq : x=sum_ainK/Z (partial(f_a)+af_s)t^a}. 
By assumption \eqref{eq : A^partial+a=0} and by uniqueness of the 
coefficients in \eqref{eq : x=sum_ainK/Z (partial(f_a)+af_s)t^a}, it follows that if 
$\partial(x)=0$ we must have $x=f_0$ with 
$\partial(f_0)=0$. By \eqref{eq : A^d^2=K} 
this implies $x=f_0\in K$.

If $x\in E_A$, we may write it in a unique way 
 as $x=\sum_{k=0}^nx_n\ell^n$, where 
$x_n\in A[t^K]$ (cf. \eqref{eq: x written in E_A}). 
Hence
\begin{eqnarray}
\partial(x)\;=\;
\Bigl(\sum_{k=0}^{n-1}
(\partial(x_k)+(k+1)x_{k+1})\ell^k\Bigr)
+\partial(x_n)\ell^n\;.
\end{eqnarray}
Again, by uniqueness of the coefficients in this 
expression, if $\partial(x)=0$ then we must 
have $\partial(x_n)=0$ and 
$\partial(x_k)=-(k+1)x_{k+1}$, for all $k=0,\ldots,n-1$. By the previous part of 
the proof $\partial(x_n)=0$ implies $x_n\in K$. This 
gives then $\partial(x_{n-1})\in K$, which implies 
$\partial^2(x_{n-1})=0$. By assumption \eqref{eq : A^d^2=K} we 
obtain $x_{n-1}\in K$. An induction then shows that 
$x_0,\ldots,x_n \in K$ and $x\in K[\ell]$ is a polynomial in $\ell$
with coefficients in $K$. It is then clear that if 
$\partial(x)=0$, then $x\in K$.

Let us now assume \eqref{eq : At^Kpartial=E_A^partial=K} 
and prove that $A$ satisfies items
 \eqref{eq : A^d^2=K} and \eqref{eq : A^partial+a=0}. 

If $f\in A$ is a non trivial solution to $\partial(f)+af=0$, then 
$\partial(ft^a)=0$ in $A[t^K]$. Since $A[t^K]^{\partial=0}=K$, it follows 
that $a=0$. This proves \eqref{eq : A^partial+a=0}.

Assume now that $f\in A$ satisfies $\partial^2(f)=0$. 
If $\partial(f)=0$, we are done because $A^{\partial=0}=K$ by assumption. 
If $\partial(f)\neq 0$, then 
$x:=f-\partial(f)\ell\in E_A$ is an element of $E_A$ 
not in $K$ and $\partial(x)=\partial(f)-(\partial^2(f)\ell+\partial(f))=0$. 
It follows that $x\in E_A^{\partial=0}=K$, and hence $f\in K$, contradicting the fact that $\partial(f)\neq 0$. The claim follows.
\end{proof}

\begin{remark}
Condition \eqref{eq : A^d^2=K} is equivalent to the 
union of the following two conditions
\begin{enumerate}
\item[(a)] $A^{\partial=0}=K$;
\item[(b)] If  $J$ is the image of the map $\partial:A\to A$, then $J\cap K=\{0\}$.
\end{enumerate}
Besides, 
it is not hard to show that \eqref{eq : A^d^2=K} also implies that
if $J'$ is the image of the map $\partial:A[t^K]\to A[t^K]$, then we have 
$J'\cap K=\{0\}$.\footnote{Indeed, An element $x$ in $J'$ can be uniquely written as 
in \eqref{eq : x=sum_ainK/Z (partial(f_a)+af_s)t^a}. 
If now $x$ belongs to $K$, by the uniqueness of the coefficients in 
\eqref{eq : x=sum_ainK/Z (partial(f_a)+af_s)t^a} only the coefficient 
corresponding to $a=0$ is non zero, and we must have 
$x=\partial(f)\in K$, for some $f\in A$. This implies $\partial^2(f)=0$, 
hence $f\in A^{\partial^2=0}=K$, which gives $f\in K$ and hence 
$0=\partial(f)=x$ as desired.}
\end{remark}

\if{\begin{remark}
We anticipate here that, although $A$ may have 
non trivial ideals stable under the action of $\partial$, 
these conditions will impose strong restrictions to the 
maps between \emph{regular singular} differential modules (cf. 
Theorem \ref{Thm. Equivalence}). 
\if{
In order to give an idea, condition (a) implies that 
for every pair $(m,n)$ of non 
negative integers and every $A$-linear 
morphism $\alpha:A^n\to A^m$ commuting with the 
action of 
$\partial$ componentwise on $A^n$ and $A^m$, the 
matrix of $\alpha$ has coefficients in $K$.
}\fi
\if{Indeed item (a) follows easily from 
the $K$-linearity of $\partial$ and the inclusion
$A^{\partial=0}\subseteq A^{\partial^2=0}=K$. Item 
(b) follows from the fact that if $x\in J\cap K$ and if 
$\partial(y)=x$ with $y\in A[t^K]$, then 
$\partial^2(y)=0$, which implies $y\in K$ and hence 
$x=0$.}\fi 
\end{remark}
}\fi

\section{The monodromy automorphism}
The exponential map over the complex numbers provides 
an exact sequence 
\begin{equation}\label{eq : exp map classical}
0\xrightarrow{\quad\qquad}\mathbb{Z}\xrightarrow{\quad\qquad}\mathbb{C}\xrightarrow[x\mapsto e^{2 i \pi x}]{}\mathbb{C}^\times\xrightarrow{\quad\qquad} 1
\end{equation}
which is extremely useful for numerous applications. 
We begin this section by proving that 
a similar exact sequence exists over any algebraically 
closed field of characteristic $0$. Notice that no 
continuity is required, the mere existence of it will be 
enough for the algebraic methods used in this paper. 
The proof uses the axiom of choice (cf. Hypothesis 
\ref{ZFC}).

\begin{lemma}\label{Lemma : basis same card of E}
Let $U$ be a $\mathbb{Q}$-vector space and $b$ a basis 
of $U$ over $\mathbb{Q}$. If the cardinality of $U$ is 
not countable, then $U$ and $b$ share the same 
cardinality. 
\end{lemma}
\begin{proof}
Let $\omega, \omega'$ be infinite ordinals such that
$\aleph_{\omega'}$ is the cardinality of $U$,  and 
$\aleph_{\omega}$ is that of $b$.
Since $b\subset U$, then $\aleph_{\omega}\leq 
\aleph_{\omega'}$. 
On the other hand, 
$U\cong\mathbb{Q}^{(b)}\subset \mathbb{Q}^b$ 
so the cardinality of $U$ is less than or equal to that of 
$\mathbb{Q}^b$ which is 
$\aleph_0^{\aleph_{\omega}}=\aleph_{\omega}$. The 
claim follows.
\end{proof}

\begin{proposition}
The additive 
group $K/\mathbb{Z}$ is isomorphic to the group $K^\times$ of invertible elements of $K$.
\end{proposition}
\begin{proof}
The torsion part of $K/\mathbb{Z}$ is 
the subgroup $\mathbb{Q}/\mathbb{Z}$ and the exact 
sequence $0\to\mathbb{Q}/\mathbb{Z}\to 
K/\mathbb{Z}\to K/\mathbb{Q}\to 0$ splits because 
$\mathbb{Q}/\mathbb{Z}$ is a divisible group (hence an injective object in the category of abelian groups). 
Analogously, $K^\times = 
\mu_\infty\oplus(K^\times/\mu_\infty)$. The abelian 
groups $V:=K/\mathbb{Q}$ and  
$W:=K^\times/\mu_\infty$ are naturally 
$\mathbb{Q}$-vector spaces because they have no 
torsion and are divisible. Since  
$\mathbb{Q}/\mathbb{Z}$ is isomorphic to 
$\mu_\infty$ it is enough to find an isomorphism 
between $V$ and $W$. We actually prove that $V$ and 
$W$ are isomorphic as $\mathbb{Q}$-vector spaces 
because they have equipotent bases. 
If the cardinality of $K$ is countable, the cardinality of $V$ and $W$ 
is at most countable. 
Moreover, they are both infinite dimensional because 
$K$ contains the algebraic closure of $\mathbb{Q}$. 
It follows that any basis of $V$ and of $W$ is countable
and $V\cong W$. 
If the cardinality of $K$ is not countable, the claim 
follows from Lemma \ref{Lemma : basis same card of E}.
\end{proof}

From now on, we now fix a group isomorphism
\begin{eqnarray}\label{eq : isom gamma def}
\overline{\gamma}\;:\;
K/\mathbb{Z}\xrightarrow{\;\;\cong\;\;}K^\times\;.
\end{eqnarray}
If $\pi:K\to K/\mathbb{Z}$ denotes the projection, we set 
$\gamma:=\overline{\gamma}\circ\pi:K\to K^{\times}$. We have an exact sequence
\begin{equation}\label{eq : exp map general}
0\xrightarrow{\quad\qquad}\mathbb{Z}\xrightarrow{\quad\qquad}K\xrightarrow{\;\;\;\;\;\;\gamma\;\;\;\;\;\;}K^\times\xrightarrow{\quad\qquad} 1
\end{equation}
In particular, we obtain so a set-theoretic bijection 
(cf. \eqref{eq : widetilde K/Z})
\begin{equation}\label{eq : abuse gamma}
\gamma\;:\;\widetilde{K/\mathbb{Z}}\xrightarrow{\;\;\cong\;\;}K^\times\;.
\end{equation}
As mentioned at the start of this section, this isomorphism resembles the classical exponential map \eqref{eq : exp map classical},\footnote{Unlike \eqref{eq : exp map general}, no topology is assumed here.} which, in the context of linear differential equations over $\mathbb{C}$, defines the \emph{monodromy automorphism}. This automorphism acts on the solution space of a differential equation via analytic continuation around the singular point $t = 0$. 

For instance, for equations like $t \frac{d}{dt}(y(t)) = a \cdot y(t)$ with $a \in \mathbb{C}$ (or $(t \frac{d}{dt})^2(y(t)) = 0$), the monodromy maps $y(t) = t^a$ to $e^{2i\pi a} t^a$, or $\log(t)$ to $\log(t) + 2i\pi$, respectively. This automorphism commutes with the derivation and belongs to the classical \emph{differential Galois group} of the equation (see \cite[Proposition 10.1]{VS}). The same reference shows that the differential Galois group of the category of differential modules over $\mathbb{C}\f{t}$ that are \emph{regular singular at $t = 0$} is the algebraic envelope of $\mathbb{Z}$, with the monodromy as its generator.

In the classical setting, this result relies on the existence theorem for Picard-Vessiot rings, which does not extend to our context of differential rings. Nonetheless, we generalize this correspondence by proving that the algebraic envelope of $\mathbb{Z}$ serves as the Tannakian group of a certain category of \emph{regular singular} differential modules over a \emph{differential $K$-algebra without exponents or logarithms} (cf. Definition \ref{Def : reg sing diff mod}). 
\if{
\begin{corollary}
If $A=K[t,t^{-1}]$, then $E_A$ is isomorphic to 
the affine $K$-algebra of the algebraic envelop of 
$\mathbb{Z}$ over $K$.
\end{corollary}
\begin{proof}
It is known that the affine algebra of the algebraic 
envelop of $\mathbb{Z}$ is given by the ring 
$K\lr{K^\times}\otimes_K K[\ell]$, where $K\lr{K^\times}$ 
denotes here the group algebra of the (abstract) 
commutative group $K^\times$ (cf. \eqref{Section : general notation}). 
The isomorphism $\overline{\gamma}$ provides then
the required isomorphism.
\end{proof}
The classical differential Galois theory deduces from this 
Corollary the fact that the category of differential 
modules over the scheme $\mathbb{G}_m$ that are 
trivialized by $E_{K[t,t^{-1}]}$ is equivalent to the 
category of $K$-linear representations of the algebraic 
envelop of $\mathbb{Z}$ (see for instance 
\cite[Proposition 10.1]{VS} together with the tool of 
Katz's canonical extension functor \cite{Katz-Can}). The 
same equivalence 
actually does not exists over any ring with derivation 
$(A,\partial)$ as showed by the conditions appearing in 
Proposition \ref{Prop : kernel of partial}.
In the following, we will obtain 
this equivalence over a conveniently general
coefficient ring $(A,\partial)$ 
(cf. Theorem \ref{Thm. Equivalence}).
}\fi

First of all, we use $\gamma$ to define an 
automorphism of $E_A$. 
Recall that $A$ is a general ring satisfying 
hypothesis \ref{Hyp : A is a diff ring}.
\begin{lemma}
There exists a unique ring automorphism $\sigma$ 
of $E_A$ satisfying:
\begin{enumerate}
\item $\sigma(f)=f$ for all $f\in A$;
\item $\sigma(t^a)=\gamma(a)t^a$ for all $a\in \widetilde{K/\mathbb{Z}}$;\footnote{With an abuse, we identify $K/\mathbb{Z}$ with $\widetilde{K/\mathbb{Z}}$.}
\item $\sigma(\ell)=\ell+1$.
\end{enumerate}
Moreover, 
\begin{enumerate}
\item $\sigma$ stabilizes $A[t^K]$ and it is automorphism of $A[t^K]$;
\item $\sigma:E_A\simto E_A$ commutes with 
$\partial:E_A\to E_A$.
\end{enumerate}
\end{lemma}
\begin{proof}
The proof is straightforward.
\if{
It is easily seen that the uniqueness of the coefficients in 
the expressions \eqref{x written as At^K} and 
\eqref{eq: x written in E_A} shows that 
$\sigma$ is a well defined $K$-linear bijection of $E_A$. 
It is also clear that it commutes with the 
the restriction of $\sigma$ to 
$A[t^K]=K[t^K]\otimes_KA$ coincides with 
$\sigma\otimes \mathrm{id}_A$. 
}\fi
\end{proof}

\begin{definition}
We call $\sigma:E_A\simto E_A$ the \emph{monodromy 
automorphism}. We set
\begin{equation}
d_\sigma\;:=\;\sigma-\mathrm{id}\;.
\end{equation}
\end{definition}
The map $d_\sigma\;:\;E_A\to E_A$ is $A$-linear and, for every $x,y\in E_A$, it satisfies
\begin{equation}
d_\sigma(xy)\;=\;d_\sigma(x)\sigma(y)+xd_\sigma(y)\;=\;
d_\sigma(x)y+\sigma(x)d_\sigma(y)\;.
\end{equation}
The following Lemma provides a set of elements 
adapted to the action of $\sigma$.
\begin{lemma}\label{Lemma : binoms generates}
Let us set for all $n\in\mathbb{Z}$, $n\geq 0$
\begin{equation}
\binom{\ell}{n}\;=\;
\frac{\ell(\ell-1)(\ell-2)\cdots(\ell-n+1)}{n!} 
\;\;\in\;\; K[\ell] \;.
\end{equation}
with the usual convention $\tbinom{\ell}{0}=1$. We extend this definition to negatives $n$ by setting $\tbinom{\ell}{n}=0$ for all $n<0$.
Then for all $n\in\mathbb{Z}$ we have
\begin{equation}\label{dsigma on binomial}
d_\sigma(\binom{\ell}{n})\;=\;
\sigma(\binom{\ell}{n})-\binom{\ell}{n}\;=\;
\binom{\ell}{n-1}\;.
\end{equation}
Moreover, every element $x\in E_A$ can be uniquely 
written as 
\begin{eqnarray}
x\;=\;\sum_{k=0}^nx_k\cdot \binom{\ell}{k}\;,
\qquad x_k\in A[t^K]\;.
\end{eqnarray}
\end{lemma}
\begin{proof}
It is classically known that 
$\{\tbinom{\ell}{k}\}_{k\geq 0}$ is a basis of the 
$K$-vector space $K[\ell]$ and verifies 
\eqref{dsigma on binomial} (cf. for instance 
\cite[Chapter 4, Section 1.1, p.162]{Robert}). The claim 
then follows from the fact that 
$E_A=A[t^K]\otimes_KK[\ell]$.
\end{proof}
\begin{proposition}\label{Prop : EAdsigma=0 A}
The following properties hold
\begin{enumerate}
\item\label{Prop : EAdsigma=0 A-3} $d_\sigma$ is surjective on $A[\ell]$ 
and $K[\ell]$ with kernels $A[\ell]^{d_\sigma=0}=A$ and
$K[\ell]^{d_\sigma=0}=K$ respectively.
\item\label{Prop : EAdsigma=0 A-1} We have $A[t^K]^{d_\sigma=0}=A$. Moreover, if $I$ denotes the image of the map 
$d_\sigma:A[t^K]\simto A[t^K]$, then
\begin{equation}\label{eq : image of d_sigma AT^K}
I\;=\;\{\sum_{a\in\widetilde{K/\mathbb{Z}}}f_at^a\;,\;\textrm{such that }f_0=0\}\;,
\end{equation}
 in particular 
$I\cap A=\{0\}$;
\item\label{Prop : EAdsigma=0 A-2} We have $E_A^{d_\sigma=0}=A$ and $d_\sigma$ is surjective on $E_A$.
\end{enumerate}
\end{proposition}
\begin{proof}
Item \eqref{Prop : EAdsigma=0 A-3} is straightforward from Lemma 
\ref{Lemma : binoms generates} and the fact that $d_\sigma$ acts trivially 
on $A$.

Let us prove item \eqref{Prop : EAdsigma=0 A-1}.
Let $x=\sum_{a\in \widetilde{K/\mathbb{Z}}}f_a 
\cdot t^a\in A[t^K]$, with $f_a\in A$. Then
\begin{equation}
d_\sigma(x)\;=\;\sum_{a\in\widetilde{K/\mathbb{Z}}}
(\gamma(a)-1)\cdot f_a\cdot t^a\;.
\end{equation}
Item \eqref{Prop : EAdsigma=0 A-1} follows from \eqref{eq : abuse gamma} and 
the uniqueness of the coefficients in this expression. 

Let us prove item \eqref{Prop : EAdsigma=0 A-2}. 
Let $x\in E_A$. 
By Lemma \ref{Lemma : binoms generates}, we can 
write $x=\sum_{k=0}^nx_k\cdot \tbinom{\ell}{k}\in 
E_A$, with $x_k\in A[t^K]$. Then
\begin{equation}\label{eq : dsigma(x)=}
d_\sigma(x)\;=\;
\Bigl(\sum_{k=0}^{n-1}(d_\sigma(x_k)+\sigma(x_{k+1}))\binom{\ell}{k}\Bigr)+d_\sigma(x_n)\binom{\ell}{n}
\end{equation}
By uniqueness of the coefficients in this 
expression, if $d_\sigma(x)=0$ then we must 
have $d_\sigma(x_n)=0$ and 
$d_\sigma(x_k)=-\sigma(x_{k+1})$ for all $k=0,\ldots,n-1$. By the previous part of 
the proof $d_\sigma(x_n)=0$ implies $x_n\in A$. This 
gives then $d_\sigma(x_{n-1})\in A$, which implies 
$d_\sigma(x_{n-1})=0$ by item \eqref{Prop : EAdsigma=0 A-1}, and hence $x_{n-1}\in A$. An induction then shows that 
$x_0,\ldots,x_n \in A$. The relation 
\eqref{eq : dsigma(x)=} becomes then 
$d_\sigma(x)=\sum_{k=0}^{n-1}\sigma(x_{k+1}) 
\binom{\ell}{k}=0$, which implies 
$x_1=x_2=\cdots=x_n=0$. Therefore $x=x_0\in A$. 
Let us prove the surjectivity of $d_\sigma$ on $E_A$. 
For every $y=\sum_{k=0}^{n-1}y_k\tbinom{\ell}{k}\in E_A$, 
we claim that there is $x=\sum_{k=0}^{n}x_k\in E_A$ such that 
$d_\sigma(x)=y$. By \eqref{eq : dsigma(x)=}, we obtain 
the following conditions on the $x_k$
\begin{equation}\label{eq : second base change}
\left\{\sm{
d_\sigma(x_{n})&=&0&&\\
d_\sigma(x_{n-1})&=&y_{n-1}&-&\sigma(x_{n})\\
d_\sigma(x_{n-2})&=&y_{n-2}&-&\sigma(x_{n-1})\\
&&&&\\
\cdots&&&&\\
d_\sigma(x_{0})&=&y_0&-&\sigma(x_{1})\;.
}
\right.
\end{equation}
For every $k=0,\ldots,n-1$ write 
$x_k=\sum_{a\in\widetilde{K/\mathbb{Z}}}g_{k,a}t^a$ and
$y_k=\sum_{a\in\widetilde{K/\mathbb{Z}}}f_{k,a}t^a$ as in 
\eqref{x written as At^K}. Let us solve the system. 
Since $A[t^K]^{d_\sigma=0}=A$, the first condition means 
$x_n\in A$, i.e. $g_{n,k}=0$ for all $k\neq 0$. 
By \eqref{eq : image of d_sigma AT^K}, we can solve 
$d_\sigma(x_{n-1})=y_{n-1}-f_{n-1,0}$, hence we can set 
$x_n=g_{n,0}:=f_{n-1,0}$ to solve the second equation. 
Moreover, since 
$A[t^K]^{d_\sigma=0}=A$, 
$x_{n-1}$ is unique up to addition of an element in $A$. 
In other words, we can freely choose $g_{n-1,0}$ in $A$. 
In particular, by \eqref{eq : image of d_sigma AT^K},  
we can choose $g_{n-1,0}$ in order that $y_{n-2}-\sigma(x_{n-1})$ 
belongs to the image of $d_\sigma$. 
Hence we can solve the second equation as well.
We see that this process can be iterated at every step and the claim 
follows. 
\end{proof}

The derivation $\partial$ does not satisfies similar rules for 
general $A$. For instance, the surjectivity of $\partial$ on $E_A$ implies 
that the category of Regular singular modules is closed by Yoneda 
extensions (cf. Proposition \ref{Prop : Reg stable by extensions}). 
\begin{proposition}\label{Prop : Ktt-1 partial epi}
The derivation $\partial$ is surjective as endomorphism of $K[\ell]$, 
$K[t,t^{-1}][\ell]$, and $E_{K[t,t^{-1}]}$. Moreover, as endomorphism of 
$K[t,t^{-1}][t^K]$, its image $J'$ is the group of of elements as in \eqref{x 
written as At^K}, where $f_0\in K[t,t^{-1}]$ has zero constant term. 
In particular, it satisfies $J'\oplus K=K[t,t^{-1}][t^K]$. For 
all the above rings the kernel of $\partial$ is $K$.
\end{proposition}
\begin{proof}
The proof is similar to that of Proposition \ref{Prop : EAdsigma=0 A}.
\end{proof}

\section{Differential modules}\label{Section : Diff mod}

A \emph{differential module} over a differential ring $(A,\partial)$ is an 
$A$-module $M$ together with a $\mathbb{Z}$-linear map
\begin{eqnarray}
\nabla\;:\;M\to M,
\end{eqnarray}
called \emph{connection},
satisfying $\nabla(fm)=\partial(f)m+f\nabla(m)$, 
for all $f\in A$ and all $m\in M$. It is usual to forget 
$\nabla$ in the notations and denote by the same symbol 
$\nabla$ all the connections, if no confusion is possible. 
A morphism $\alpha:M\to N$ of differential modules is 
an $A$-linear morphism commuting with the 
connections of $M$ and $N$ respectively. A morphism is 
also called sometimes  \emph{horizontal morphism} in order to distinguish it from $A$-linear morphisms.
The set of morphisms of differential modules is denoted 
by $\Hom_A^\nabla(M,N)$. It is naturally a 
$A^{\partial=0}$-module. Let us denote 
by $\partial-\Mod(A)$ the category of differential 
modules over $A$. 
The \emph{kernel} and 
\emph{cokernel} of a morphism of differential modules 
are naturally differential modules. More generally 
all the constructions of 
linear algebra exist for differential modules and the  
category $\partial-\Mod(A)$ is abelian. 
We quickly list here the notions that we use in this paper. \if{
The reader may find a 
more complete description in 
\cite[Section 5.3]{Kedlaya-Book}.
}\fi
\begin{itemize}
\item A sequence of differential 
modules is \emph{exact} if the underling sequence of $A$-modules is.

\item The \emph{direct sum} 
of $(M,\nabla)$ and $(N,\nabla)$ is the 
$A$-module $M\oplus N$ together with the connection 
$\nabla(m,n):=(\nabla(m),\nabla(n))$, $m\in M$, $n\in N$.
\item The \emph{tensor product} $M\otimes N$ 
of differential modules is the $A$-module 
$M\otimes_A N$ endowed with
the connection $\nabla(m\otimes n)=\nabla(m)\otimes n+m\otimes\nabla(n)$, for $m\in M$ and $n\in N$.

\item The elements of $M$ that are 
in the kernel of $\nabla$ are called 
\emph{solution of $M$ with value in $A$}. More 
generally, if $(B,\partial')$ is another differential ring and 
if a ring morphism $A\to B$ commuting with the 
derivations is given, then the tensor product 
$M\otimes_AB$ (with connection 
$\nabla\otimes1+1\otimes\partial'$) is naturally a 
differential module over $B$. We say that \emph{the 
solutions of $M$ with value in $B$ are the elements in 
the kernel of the connection of $M\otimes_A B$.}
This defines a functor 
\begin{equation}
\Sol_B\;:\;\partial-\Mod(A)\to\Mod(B^{\partial=0})
\end{equation}
associating to every differential module $M$ over $A$ the 
$B^{\partial=0}$-module of the solutions of $M$ with 
values in $B$. Solutions $x\in M\otimes_AB$ correspond bijectively to horizontal morphisms $B\to M\otimes_AB$. Therefore we have 
an identity of functors 
\begin{equation}
\Sol_B(M)\;=\;\Hom_B^\nabla(B,M\otimes_A B)\;.
\end{equation}
It is an additive functor and if $B$ is flat over $A$ it is 
also left exact. It is also left 
exact on exact sequences formed 
by differential modules that are flat as $A$-modules.

\item The differential module $(A,\partial)$ 
is called \emph{the trivial object} or the \emph{unit 
object}, we denote it by $A$ or by $\textbf{1}$. More 
generally we say that an object is 
\emph{trivial} if it is isomorphic to a direct sum of 
copies of $(A,\partial)$.

If $A\to B$ is a morphism of differential rings as above, 
we say that a differential module $M$ over $A$ is 
\emph{trivialized by $B$} if $M\otimes_A B$ is trivial as 
a differential module over $B$.

For every \emph{trivial} object $M$ over $B$ we have a natural isomorphism of differential modules over $B$
\begin{equation}\label{eq : trivial downup}
M^{\nabla=0}\otimes_{B^{\partial=0}}B\;\xrightarrow{\;\sim\;}\;M\;,
\end{equation}
where the connection on the left side module is $1\otimes\partial$.

\item The $A$-module of \emph{all} $A$-linear maps 
$\Hom_A(M,N)$ has a structure of differential module 
equipped with the connection $\nabla(f)=f\circ \nabla - 
\nabla\circ f$. Therefore, the horizontal morphisms $\Hom_A^\nabla(M,N)$ form a 
$A^{\partial=0}$-module which is nothing but the 
solutions with values in 
$A$ of the differential module $\Hom_A(M,N)$.
\item The \emph{dual module} $M^*$ is the differential 
module $\Hom_A(M,A)$.
\item We denote by $\H^1(M)$ the cokernel of $\nabla$. 
\item 
We denote by $\Ext(M, N)$ the group of Yoneda extensions. Recall that the elements of $\Ext(M, N)$ are equivalence classes of exact sequences of the form 
$0 \to N \to P \to M \to 0$
where two sequences 
$
0 \to N \to P \to M \to 0 $ and $0 \to N \to P' \to M \to 0
$
are considered equivalent if there exists an isomorphism $P \cong P'$ that induces the identity maps on $M$ and $N$. The group $\Ext(M, N)$ can be equipped with the Baer sum, which endows it with a group structure in the usual way (cf. \cite[Section 3.4]{Weibel-book}).
\item If $M$ is a finite projective $A$-module, 
we have an isomorphism of groups 
(cf. \cite[Lemma 5.3.3 and Remark 
5.3.4]{Kedlaya-Book}) 
\begin{eqnarray}\label{eq : yoneda ext}
\H^1(M^*\otimes N)=\Ext(M,N)\;.
\end{eqnarray} 
\end{itemize}
\subsection{Matrix of the connection}
Assume that $M$ be a differential module which is finite 
free as $A$-module with basis $b=(b_1,\ldots,b_n)$.  
For all $i=1,\ldots,n$ we can express $\nabla(b_i)$ in a unique way as 
$\nabla(b_i)=\sum_{j=1}^ng_{j,i}\cdot b_j$, with $g_{j,i}\in A$. 
We call the square matrix 
$G=(g_{i,j})_{i,j=1,\ldots,n}\in M_{n\times n}(A)$ the matrix 
of the connection of $M$. If $A^n\simto M$ is the 
isomorphism sending the canonical basis of $A^n$ into 
the basis $b$ we may identify the vectors
$m=\sum_{i=1}^n f_ib_i$ of $M$ with 
column $n$-uples $\vec{f}:=\phantom{}^t(f_1,\ldots,f_n)\in A^n$ and the 
function $\nabla$ is given by
\begin{equation}\label{eq : matrix of conn}
\nabla\left(\sm{f_1\\\vdots\\\\f_n}\right)\;=\;\left(\sm{\partial(f_1)\\\vdots\\\\\partial(f_n)}\right)+G\cdot
\left(\sm{f_1\\\vdots\\\\f_n}\right)\;.
\end{equation}
The matrix $G$ determines then entirely the map 
$\nabla$. Reciprocally every square matrix with 
coefficients in $A$ defines a connection 
by the rule \eqref{eq : matrix of conn}. 

Let us now consider a base change in $M$ where the 
vectors in the new basis are written as $H\vec{f}$, with 
$H\in GL_n(A)$ then the matrix of the connection in the 
new basis is given by
\begin{eqnarray}\label{eq : base change rule}
\partial(H)H^{-1}+HGH^{-1}\;,
\end{eqnarray}
where $\partial(H)$ means the matrix obtained from $H$ 
by differentiating every coefficient. \\

Let $0\to M_1\to M_2\to M_3\to 0$ be an exact 
sequence of differential modules over $A$ that are 
finite free as $A$-modules. Let $b_1\subset M_1$ and 
$b_3\subset M_3$ be two basis and let $G_1$ and 
$G_3$ be the corresponding matrices of the connections 
$\nabla_1$ and $\nabla_3$ respectively. 
Let now $\widetilde{b}_3\subseteq M_2$ be a set of 
vectors which are in bijection with $b_3\subset M_3$ 
under the projection $M_2\to M_3$. Then the set 
$b_2:=b_1\cup\widetilde{b}_3$ is easily seen to be a 
basis of the $A$ module $M_2$ an the matrix $G_2$ of 
$\nabla_2$ in this basis is a block matrix of  the form
\begin{eqnarray}\label{eq: exact sequence}
G_2\;=\;\left(\sm{G_1&&*\\\\0&&G_3}\right)\;.
\end{eqnarray}
Reciprocally, if the matrix of a finite free differential 
module has this form, then the module fits in the middle 
of an exact sequence of differential modules as above.

\begin{definition}\label{Definition : N(a)}
Assume now that $A$ is a $K$-algebra with  a $K$-linear derivation 
$\partial:A\to A$. For all $a\in K$ we denote by 
\begin{eqnarray}
N(a)
\end{eqnarray}
the one dimensional finite free differential module 
whose matrix in some basis is $G=a$.
\end{definition}
\begin{example}\label{Example : Jordan}
Let $M$ be a finite free differential module over $A$ with matrix $G \in M_n(K)$. Changing basis by $H \in GL_n(K)$ transforms $G$ into $HGH^{-1}$ by the base change rule \eqref{eq : base change rule}. Thus, we can choose a basis where the connection matrix is triangular with eigenvalues $a_1, \ldots, a_n \in K$. By \eqref{eq: exact sequence}, this implies $M$ has free one-dimensional sub-quotients isomorphic to $N(a_1), \ldots, N(a_n)$.
\end{example}
We now analyze the effect of change of derivation in 
$A$.
\begin{proposition}\label{Prop: change of deriv}
Let $(A,\partial)$ be a differential ring. 
For all $h\in A$ the map $h\cdot \partial$ is another 
derivation on $A$. Then
\begin{enumerate}
\item If  $(M,\nabla)$ is a  differential module over 
$(A,\partial)$, then $(M,h\nabla)$ is a differential module over $(A,h\partial)$;
\item If $\alpha:(M,\nabla)\to (N,\nabla')$ 
is an $A$-linear morphism commuting with the 
connections $\nabla$ and $\nabla'$, then 
$\alpha$ also commutes with $h\nabla$ and $h\nabla'$.
\end{enumerate}
We then have a functor
\begin{equation}
C_\partial^{h\partial}\;:\;\partial-\Mod(A)\;\xrightarrow{\qquad}\;
(h\partial)-\Mod(A)\;.
\end{equation}
Associating to $(M,\nabla)$ the module $(M,h\nabla)$
and which is the identity on the morphisms.

If $h\in A^{\times}$ this functor is an equivalence of categories.\hfill$\Box$
\end{proposition}

\subsection{Scalar extension of the base ring}
\label{base change}
Let $(B,\partial_B)$ be another differential ring and $f:A\to 
B$ be a ring morphism. We assume that there exists an 
invertible element $h\in B^{\times}$ such that 
\begin{equation}
h\cdot\partial_B\circ f \;=\; f\circ\partial
\end{equation}
that is the derivation $h\partial_B$ 
extends the derivation $\partial$ of $A$. 

Let $(M,\nabla)$ be a differential module over 
$(A,\partial)$. 
The scalar extension $M\otimes_AB$ is naturally a 
differential module over $(B,h\partial_B)$ endowed with 
the connection 
\begin{equation}
\nabla':=\nabla\otimes\mathrm{Id}_B+\mathrm{Id}_M\otimes(h\partial_B)\;.
\end{equation}

\begin{definition}
The scalar extension of $M$ to $(B,\partial_B)$ is
the differential module 
\begin{equation}
C_{h\partial_B}^{\partial_B}(M,\nabla')\; = \;
(M,h^{-1}\nabla')\;.
\end{equation}
\end{definition}

\section{Semi-linear representations}
Let $G$ be a group. We recall that a 
\emph{$K$-linear representation of $G$} is a $K$-vector 
space $V$ together with a group homomorphism 
\begin{equation}\label{eq : rho def representation}
\rho:G\to \Aut_K(V)\;,
\end{equation} 
where $\Aut_K(V)$ is the group 
of $K$-linear automorphisms of $V$. For every $g\in G$ 
and $v\in V$ we set $g(v):=\rho(g)(v)$.
A \emph{morphism of $K$-linear representations} 
$\alpha:V\to W$ is a $K$-linear map commuting with the 
actions of $G$ on $V$ and $W$. We denote by 
$\Hom_K^G(V,W)$ the $K$-vector space of morphisms.
We denote by $\Rep_K(G)$ the category of 
$K$-linear representations and by 
\begin{equation}
\Rep_K^{\mathrm{fin}}(G)
\end{equation}
the full sub-category formed by representations of $G$ 
that are finite dimensional over $K$. 
If $V\in\Rep(G)$ we denote by 
$V^G\subseteq V$ the $K$-vector space of vectors $v\in 
V$ such that $g(v)=v$ for all $g\in G$.

Let now $E$ be a commutative 
$K$-algebra together with a group homomorphism $G\to 
\Aut_K(E)$, where $\Aut_K(E)$ is the group of 
$K$-linear ring automorphisms of $E$.  We call $E$ 
a \emph{$G$-algebra over $K$}.
An \emph{$E$-semi-linear representation of $G$} or 
simply a \emph{$G$-module over $E$} 
is an $E$-module $V$ 
together with a structure of $K$-linear representation of 
$G$ satisfying for all $g\in G$, $e\in E$, $v\in V$
\begin{equation}
g(e\cdot v)=g(e)g(v)\;.
\end{equation}
A \emph{morphism} $\alpha:V\to W$ 
between $E$-semi-linear representations of $G$ is an $E$-linear 
map commuting with the action of $G$. We denote by 
$\Hom_E^G(V,W)$ the group of morphisms. 
It is naturally a $E^G$-module, where $E^G$ is the 
sub-ring of $E$ formed by elements $e\in E$ 
satisfying $g(e)=e$ for all  $g\in G$. 
We denote by $\Rep_E(G)$ the category of $G$-modules 
over $E$. Again, all the operations of linear 
algebra exists we only list those that we use in this paper.
\begin{itemize}
\item A sequence of $G$-modules over $E$ 
is \emph{exact} if the underling sequence of 
$E$-modules is.

\item If $V$ and $W$ are two $G$-modules over $E$, 
their \emph{direct sum} is the $E$-module $V\oplus W$ 
together with the action of $G$ given by 
$g(v,w):=(g(v),g(w))$, $v\in V$, $w\in W$, $g\in G$.
\item The \emph{tensor product} $V\otimes W$ 
of $G$-modules over $E$ is the $E$-module 
$V\otimes_E W$ endowed with
the group action 
$g(v\otimes w)=g(v)\otimes g(w)$, for $v\in M$, $w\in N$, $g\in G$.

\item Let $F$ be another $G$-algebra over $K$. 
If $E\to F$ is a morphism of $K$-algebras 
commuting with the action of $G$ we have a 
\emph{scalar extension} functor 
$V\mapsto V\otimes_EF$. Namely, the 
module $V\otimes_EF$ is naturally a $G$-module over 
$F$. We have an equality of functors
\begin{eqnarray}
(V\otimes_EF)^G\;=\;\Hom_F^G(F,V\otimes_EF)\;.
\end{eqnarray}
It is an additive functor and if $F$ is flat over $E$ it is 
also left exact. It is also left exact 
on exact sequences formed 
by $G$-modules over $E$ that are flat as $E$-modules.
\item The $G$-algebra $E$   
is called \emph{the trivial $G$-module} 
or the \emph{unit object}, we denote it by $E$ or by 
$\textbf{1}$. More 
generally we say that an object is 
\emph{trivial} if it is isomorphic to a direct sum of 
copies of $E$.

If $E\to F$ is a morphism of $G$-algebras over $K$ 
as above, we say that a $G$-module $V$ over $E$ is 
\emph{trivialized by $F$} if $V\otimes_E F$ is trivial as 
a differential module over $F$.

For every \emph{trivial} $G$-module $V$ over $F$ 
we have an $F$-linear isomorphism of $G$-modules
\begin{equation}\label{eq: trud}
V^G\otimes_{F^{G}}F\;\xrightarrow{\;\sim\;}\;
V\;.
\end{equation}

\item The $E$-module of \emph{all} $A$-linear maps 
$\Hom_E(V,W)$ has a structure of $G$-module 
given by $g(f)=g\circ f\circ g^{-1}$. Therefore, the  
morphisms $\Hom_E^G(V,W)$ is the
$E^{G}$-module $\Hom_E(V,W)^G$.
\item The \emph{dual module} $V^*$ is the differential 
module $\Hom_E(V,E)$.
\end{itemize}
\subsection{Matrix of a semi-linear representation of 
$\mathbb{Z}$}
\label{Section: matrix of 1}
If $G=\mathbb{Z}$, we use the notation $[n]$ for the 
elements of $G$ in order to distinguish the 
action of $G$ from the multiplication by the 
elements of $\mathbb{Z}\subseteq K$.
A $K$-linear representation of $\mathbb{Z}$ is then completely 
determined by the action of the individual $K$-linear 
automorphism $[1]:V\to V$.  A $K$-linear 
representation of $\mathbb{Z}$ is nothing but a vector 
space together with a $K$-linear automorphism. We call 
$[1]$ the \emph{monodromy operator} and we will 
indicate it sometimes with the symbol
\begin{eqnarray}
\sigma\:=[1]\;.
\end{eqnarray}
Often, a $K$-linear representation of $\mathbb{Z}$ is 
also called $\sigma$-module over $K$.  
If $V\in\Rep_{K}^{\mathrm{fin}}(\mathbb{Z})$,  these 
objects are completely classified by the classical Jordan 
normal form of the endomorphism $\sigma$. 
Namely, in a hand $V$ admits a 
Jordan-Hölder series formed by rank one 
$\sigma$-modules. 
On the other hand, if $V_1,V_2$ are two rank one 
$\sigma$-modules, then the Yoneda extension group 
$\Ext(V_1,V_2)$ has dimension $1$ if $V_1\cong V_2$ 
and it is zero dimensional otherwise.
Namely, the 
monodromy automorphism $\sigma:V\to V$ 
admits a Jordan normal form. 
We can find a basis of $V$ in which the matrix of the 
monodromy operator $\sigma$ has matrix in Jordan canonical square blocks of the form 
\begin{equation}\label{eq : J(lambda,n)}
J(\lambda,n)\;:=\;\underbrace{\left(\sm{
\lambda&1&0&0&&\cdots&&0\\
0&\lambda&1&0&&\cdots&&0\\
0&0&\lambda&1&&\cdots&&0\\
&&&&\ddots\!\!&&\phantom{\cdots}&\vdots\\
&&&&&\phantom{\cdots}&\phantom{\cdots}&\phantom{\cdots}\\
\cdots&\cdots&\cdots&\cdots&\cdots&\lambda&1&0\\
\cdots&\cdots&\cdots&\cdots&\cdots&0&\lambda&1\\
\cdots&\cdots&\cdots&\cdots&\cdots&0&0&\lambda
}
\right)}_{n\textrm{ columns}}
\end{equation}
The monodromy operator $\sigma:E_A\to E_A$ makes 
$E_A$ a $\sigma$-module (or equivalently a 
$G$-algebra over $K$, with $G=\mathbb{Z}$). The 
following Lemma uses the Jordan normal form to prove 
that finite $\varphi$-modules 
are all trivialized by the ring $E_A$. 
\begin{proposition}\label{Lemma : every V is trivialized}
Every $V\in\Rep_K^{\mathrm{fin}}(\mathbb{Z})$ is 
trivialized by $E_A$.
\end{proposition}
\begin{proof}
Without loss of generality 
we may assume 
that $V$ is indecomposable 
(i.e. $V$ is not direct sum of sub-objects). 
Let $b_1,\ldots,b_n\in V$ be a basis in 
which the matrix $J\in GL_n(K)$ of the monodromy 
operator $[1]$ is in Jordan canonical form. Since $V$ is 
indecomposable, we have 
$J=J(\lambda,n)$
\if{
=\lambda
\cdot\mathrm{Id}+N_n$, where $N_n$ is 
the standard nilpotent matrix\footnote{i.e. 
$N_n=(\epsilon_{i,j})_{i,j=1,\ldots,n}$, where 
$\epsilon_{i,j}=0$ for all $i\neq j+1$ and 
$\varepsilon_{i,j+1}=1$.}  
}\fi
(cf. \eqref{eq : J(lambda,n)}).
In other words we have
\begin{equation}\label{eq : first base change}
\left\{
\sm{[1](b_1)&=&\lambda\cdot b_1&\\
\\
[1](b_k)&=&\lambda\cdot b_k&+&b_{k-1}&\textrm{ for all }k=2,\ldots,n.
}
\right.
\end{equation}
The action of $[1]$ in 
the basis $b_i\otimes 1\in V\otimes_KE_A$ is again $J$. 
Let $a\in\widetilde{K/\mathbb{Z}}$ be an element such 
that $\gamma(a)=\lambda^{-1}$ 
(cf.  \eqref{eq : abuse gamma}). 
We may now perform a first base change replacing every 
$b_i\otimes 1$ by $b_i':=t^{a}(b_i\otimes 1)$. Then
\begin{equation}\label{eq : first base change-2}
\left\{
\sm{[1](b_1')&=&b_1'&\\
\\
[1](b_k')&=&b_k'&+&\lambda^{-1}b_{k-1}'&\textrm{ for all }k=2,\ldots,n.
}
\right.
\end{equation}
The new matrix of $[1]$ being $\lambda^{-1}J(\lambda,n)$,
with a base change by a matrix in $GL_n(K)$ 
(which preserves the eigenvalues of $[1]$) we can 
find another basis of $V\otimes E_A$ of the 
form $J(1,n)$.
In other words, we can assume $\lambda=1$ in \eqref{eq : first base change-2}.

We now perform another base change after which the 
matrix of $[1]$ will be the identity. More precisely, 
we will prove that there exists a new basis $b_1'',\ldots,b_n''$ of $V\otimes_KE_A$ of 
the form 
\begin{eqnarray}
b_k''=x_{k,k}b_k'+x_{k,k-1}b_{k-1}'+x_{k,k-2}b_{k-2}'+\cdots+x_{k,1}b_{1}',\quad\textrm{for all }k=1,\ldots,n
\end{eqnarray}
where
\begin{enumerate}
\item for all $k=1,\ldots,n$ one has $x_{k,k}=1$;
\item for all $i,j$ the coefficients $x_{i,j}$ belong to $K[\ell]$ 
\item for all $k=1,\ldots,n$ one has $[1](b_k'')=b_k''$.
\end{enumerate}
By convenience of notation, for all $k=1,\ldots,n$ 
we set $b_0'=0$. The fact that 
$x_{k,k}=1$ implies that for every choice of $x_{i,j}$ 
the family $\{b_k''\}_{k=1,\ldots,n}$ is  a basis of 
$V\otimes_KE_A$.
With this setting, for all $k=1,\ldots,n$, one has
\begin{eqnarray}
[1](b_k'')-b_k''
&\;=\;&\sum_{i=1}^k\sigma(x_{k,i})
(b_{i}'+b_{i-1}')-\sum_{i=1}^kx_{k,i}b_i'\\
&\;=\;&\sum_{i=1}^k(\sigma(x_{k,i})-x_{k,i})b_i'
+\sum_{i=1}^k\sigma(x_{k,i})b_{i-1}'\\
&\;=\;&\sum_{i=1}^{k}(\sigma(x_{k,i})-x_{k,i})b_i'
+\sum_{j=1}^{k-1}\sigma(x_{k,j+1})b_{j}'\\
&\;=\;&0\cdot b_k'+\sum_{i=1}^{k-1}(\sigma(x_{k,i})-x_{k,i}+\sigma(x_{k,i+1}))b_i'
\end{eqnarray}
which provides for all $k=2,\ldots,n$ the system of conditions
\begin{equation}\label{eq : second base change}
\left\{\sm{
d_\sigma(x_{k,k-1})&=&-\sigma(x_{k,k})&=&-1\\
d_\sigma(x_{k,k-2})&=&-\sigma(x_{k,k-1})\\
d_\sigma(x_{k,k-3})&=&-\sigma(x_{k,k-2})\\
&&\\
\cdots&&\\
d_\sigma(x_{k,1})&=&-\sigma(x_{k,2})\;.
}
\right.
\end{equation}
The claim then follows from the fact that 
$d_\sigma:K[\ell]\to K[\ell]$ is surjective (cf. Lemma 
\ref{Lemma : binoms generates}).
\end{proof}

\section{The Monodromy functor}
Let $M$ be a differential module over $A$. 
The module $M\otimes_AE_A$ carries the connection 
\begin{eqnarray}
\nabla=\nabla\otimes 1+1\otimes\partial
\end{eqnarray} 
and also the 
semi-linear automorphism $1\otimes\sigma$. Therefore, 
the kernel of the connection 
$(\M\otimes_AE_A)^{\nabla=0}=\Sol_{E_A}(M)$ 
is a $K$-vector space with an 
action of $1\otimes\sigma$. We define an action of 
$\mathbb{Z}$ on $(\M\otimes_AE_A)^{\nabla=0}$ as
\begin{eqnarray}
[n]\cdot x\;:=\;(1\otimes\sigma)^n(x)\;,\qquad 
n\in\mathbb{Z}\;,\quad x\in 
(\M\otimes_AE_A)^{\nabla=0}\;.
\end{eqnarray}
We use the notation $[n]$ in order to distinguish this 
action of $\mathbb{Z}$ 
from the multiplication by the elements of 
$\mathbb{Z}\subseteq K$. The $K$-vector space $\Sol_{E_A}(M)$ together with this action of $\mathbb{Z}$ is an object of 
$\Rep_K(\mathbb{Z})$ that we denote by 
\begin{eqnarray}
\Mon(M)\;. 
\end{eqnarray}
We call it the \emph{monodromy} of $(M,\nabla)$. It is 
clear that 
\begin{equation}
\Mon\;:\;\partial-\Mod(A)\;\xrightarrow{\qquad}\;\;
\Rep_{K}(\mathbb{Z})\;
\end{equation}
is an additive functor. It is left exact because $E_A$ is 
a free $A$-module.
The individual operator $[1]\in\mathbb{Z}$ 
acting on $\Mon(M)$ 
is called \emph{the monodromy operator}. It is clear that 
the monodromy operator determines the entire action of 
$\mathbb{Z}$ as $[n]=[1]^n$ for all $n\in\mathbb{Z}$. 

We now construct a functor in the other direction
\begin{eqnarray}
\Rm\;:\;\Rep_{K}(\mathbb{Z})
\;\xrightarrow{\qquad}\;\partial-\Mod(A)\;.
\end{eqnarray}

Let $V$ be a $K$-linear representation of 
$\mathbb{Z}$. The $E_A$-module $V\otimes_KE_A$ is a 
$K$-semi-linear representation of $\mathbb{Z}$ with the 
action $[n](v\otimes x):=[n](v)\otimes \sigma^n(x)$. 
Moreover, it carries a connection $1\otimes\partial$ 
which commutes with the action of $\mathbb{Z}$, 
because $\partial$ commutes with $\sigma$ as 
endomorphisms of $E_A$. 
Therefore, the $A$-module 
\begin{equation}
\Rm(V)\;=\;
(V\otimes_KE_A)^{\mathbb{Z}}=\mathrm{Ker}([1]-1\;:
\;V\otimes_KE_A\to V\otimes_KE_A)
\end{equation}
formed by the elements of $V\otimes_KE_A$ 
fixed by the action of $\mathbb{Z}$ carries the 
connection $1\otimes\partial$ and it is a differential 
module over $A$. 

\begin{remark}
\label{dependence on MON and Rm the 
choices}
The definition of the functors $\Mon$ and $\Rm$ depends on 
the following choices 
\begin{enumerate}
\if{
\item A coordinate $t$ of $K[t,t^{-1}]$; and 
hence of the derivation 
$\partial=t\frac{d}{dt}$ on $K[t,t^{-1}]$ (the unique 
$K$-linear derivation of $K[t,t^{-1}]$ such that $\partial(t)=t$);
}\fi
\item The choice of the structural injective 
morphism $\rho:K[t,t^{-1}]\to A$;
\item The choice of a derivation $\partial:A\to A$ extending 
$t\frac{d}{dt}:K[t,t^{-1}]\to K[t,t^{-1}]$;
\item The choice of an isomorphism 
$\gamma:K/\mathbb{Z}\simto K^\times$.
\end{enumerate}
\end{remark}
\begin{notation}\label{Notation : mon_rho}
When necessary, we will denote $\Mon$ as
\begin{equation}
\Mon=\Mon_{\rho}=\Mon_{\rho,\partial}=\Mon_{\rho,\partial,\gamma}\;.
\end{equation}
The same notation will be used for $\Rm$.
\end{notation}
The following Lemma gives a first general link between internal $\Hom$, internal $\otimes$ and Dual.
\begin{lemma}\label{Lemma : Can morphisms M}
For all $M,N\in\partial-\Mod(A)$ 
one has functorial $K$-linear homomorphisms of 
$K$-linear representations of $\mathbb{Z}$
\begin{eqnarray}
\Mon(\Hom_A(M,N))&\;\to\;&\Hom_K(\Mon(M),\Mon(N))\label{eq:arrowonhom}\\
\Mon(M\otimes_A N)&\leftarrow&\Mon(M)\otimes_K\Mon(N)
\label{eq:arrowonhom-2}\\
\Mon(M^*)&\to&\Mon(M)^*\label{eq:arrowonhom-3}\;.
\end{eqnarray}

Analogously, for all $V,W\in\Rep_{K}(\mathbb{Z})$ 
one has functorial $A$-linear homomorphisms of 
differential modules over $A$
\begin{eqnarray}
\Rm(\Hom_K(V,W))&\;\to\;&\Hom_A(\Rm(V),\Rm(W))\\
\Rm(V\otimes_KW)&\leftarrow&\Rm(V)\otimes_A\Rm(W)\\
\Rm(V^*)&\to&\Rm(V)^*\;.
\end{eqnarray}
\end{lemma}
\begin{proof}
Let us prove the existence of \eqref{eq:arrowonhom}.
We have a canonical diagram
\begin{equation}
\!\!\!\!
\!\!\!\!
\xymatrix{\Hom_A(M,N)\otimes_A E_A\ar[r]&\Hom_{E_A}(M\otimes_{A}E_A,N\otimes_AE_A)\\
(\Hom(M,N)\otimes_AE_A)^{\nabla=0}
\ar[r]\ar@{}|-{\cup}[u]&
(\Hom_{E_A}(M\otimes_{A}E_A,N\otimes_AE_A))^{\nabla=0}
\ar@{}|-{\cup}[u]\ar[d]\\
&\Hom_{K}((M\otimes_{A}E_A)^{\nabla=0},(N\otimes_AE_A)^{\nabla=0})}
\end{equation}
where the first line from the top is the canonical one 
(cf. \cite[\S5, N.3]{Bou-Alg-II}), and it is injective because $E_A$ is free 
as $A$-module \cite[\S5, N.3, Proposition 7]{Bou-Alg-II}. 
It is not hard to prove that this morphism commutes with the actions of 
$\nabla$ and $[1]$. The second line of the diagram is just the 
restriction of the first-one to the $E_A^{\partial=0}$-modules of solutions, 
since the connections commute with the actions of $[1]$, the objects in the 
second line inherit an action of $[1]$.
Now, the vertical arrow associates to a map 
$\alpha:\Hom_{E_A}(M\otimes_{A}E_A,N\otimes_AE_A)$ commuting 
with the connections (but possibly not with $[1]$) its restriction to the 
spaces of solutions 
$(M\otimes_{A}E_A)^{\nabla=0}$ and 
$(M\otimes_{A}E_A)^{\nabla=0}$. It is not hard to show that it commutes 
with the action of $[1]$.

The arrow \eqref{eq:arrowonhom-3} is a consequence of 
\eqref{eq:arrowonhom} and the fact that $\Mon(\bs{1})=\bs{1}$.

Let us prove the existence of \eqref{eq:arrowonhom-2}. 
We have a canonical map 
$
(M\otimes_AE_A)\otimes_K(N\otimes_AE_A)\to
(M\otimes_AE_A)\otimes_{E_A}(N\otimes_AE_A)\cong
(M\otimes_AN)\otimes_AE_A
$.
This map commutes with the connections 
and the actions of $[1]$. Taking the fixed points, it induces a map
$
(M\otimes_AE_A)^{\nabla=0}\otimes_K(N\otimes_AE_A)^{\nabla=0}\to
((M\otimes_AE_A)\otimes_{E_A}(N\otimes_AE_A))^{\nabla=0}
\cong
((M\otimes_AN)\otimes_AE_A)^{\nabla=0}
$
which is \eqref{eq:arrowonhom-2}.

The proof for the analogous properties of $\Rm$ is similar.
\end{proof}
\begin{example}\label{Exemple: Vlambda N(a)}
Let $a\in\widetilde{K/\mathbb{Z}}$ and 
$\lambda:=\gamma(-a)\in K^\times$. 
Denote by 
\begin{equation}\label{eq : Vlambda}
V_\lambda
\end{equation}
the one dimensional $\sigma$-module over $K$ such that $\sigma:V\to V$ 
is the multiplication  by $\lambda$. Then 
(cf. Définition \ref{Definition : N(a)})
\begin{eqnarray}\label{eq : MonNa=Vlambda}
\Mon(N(a))\;=\;V_\lambda
\qquad \textrm{ and }\qquad
\Rm(V_\lambda)\;=\;N(a)\;.
\end{eqnarray}
To see this it is enough to follows the first lines of 
the proof of Proposition \ref{Lemma : every V is trivialized}. 
If $b\in N(a)$ is the basis such that $\nabla(b)=ab$, then 
$(N(a)\otimes_AE_A)^{\nabla=0}=K\cdot(b\otimes t^{-a})$ and the action 
of $1\otimes\sigma$ on $b\otimes t^{-a}$ is the multiplication by 
$\gamma(-a)$.
\end{example}

\section{Regular singular differential modules}

The following two Lemmas show the interest of Definition \ref{Def : algebras without exp nor log}. 
\begin{lemma}\label{Lemma: up down}
Let $A$ be a differential $K$-algebra 
without exponents nor logarithm. 
If $V\in\Rep_K(\mathbb{Z})$ one has 
\begin{equation}\label{eq : VEnabla=V}
(V\otimes_KE_A)^{1\otimes\partial=0}\;=\;V\;.
\end{equation}
Analogously, if $M\in\partial-\Mod(A)$ is free as $A$-module, then
\begin{equation}\label{eq : MEsigma=M}
(M\otimes_AE_A)^{1\otimes\sigma=1}\;=\;M\;.
\end{equation}
\end{lemma}
\begin{proof}
If $\{v_i\}_i\subset 
V$ is a basis, then $(v_i\otimes1)_i$ is a basis of 
$V\otimes_KE_A$ and every $x\in V\otimes_KE_A$ can 
be uniquely written as $x=\sum_{i=1}^nv_i\otimes x_i$, 
with $x_i\in E_A$. Then, $(1\otimes\partial)(x)=0$ if 
and only if $\partial(x_i)=0$ for all $i$. By 
Proposition \ref{Prop : kernel of partial}, it follows that 
$x_i\in K$ for all $i$ and $x\in 
V\otimes_KK=V$. 
Similarly, the claim about $M$ follows from Proposition \ref{Prop : EAdsigma=0 A}.
\end{proof}
The following Lemma says that $\Mon$ and $\Rm$ commute with the 
fiber functors.
\begin{lemma}
\label{Lemma: Mon and R commutes to solutions}
Let $A$ be a differential $K$-algebra 
without exponents nor logarithm.
For every $M\in\partial-\Mod(A)$ we have a functorial injective $A^{\partial=0}$-linear inclusion 
\begin{equation}
M^{\nabla=0}\;\subseteq\;\Mon(M)^{\mathbb{Z}}\;.
\end{equation}
This map is an isomorphism if $M$ is free as 
$A$-module.

Let $V\in\mathrm{Rep}_{K}^{\mathrm{fin}}(\mathbb{Z})$, then 
\begin{equation}
\Rm(V)^{\nabla=0}\;\cong\;V^{\mathbb{Z}}\;.
\end{equation}
\end{lemma}
\begin{proof}
Since $E_A$ is a free $A$-module, we have an inclusion 
$M\subseteq M\otimes_AE_A$, given by  $m\mapsto m\otimes 1$, which is $A$-linear and 
commutes with the connections. The action of 
$[1]:=1\otimes\sigma$ on $M\otimes_AE_A$ stabilizes 
the image of $M$ and we then have 
\begin{eqnarray}\label{eq :153}
\nonumber
M^{\nabla=0}\;=\;
((M\otimes_AE_A)^{[1]=1})^{\nabla=0}\subseteq 
(M\otimes_AE_A)^{\nabla=0}\;=\;\Mon(M)\;.
\end{eqnarray}
Now, the operator 
$[1]:=1\otimes\sigma$ acting on $M\otimes_AE_A$ commutes with 
$\nabla:=\nabla\otimes1+1\otimes\partial$. 
Therefore we have 
$\Mon(M)^{\mathbb{Z}}:=((M\otimes_AE_A)^{\nabla=0})^{[1]=1}=
((M\otimes_AE_A)^{[1]=1})^{\nabla=0}
\supseteq M^{\nabla=0}$. 
By Lemma \ref{Lemma: up down}, if $M$ is a free 
$A$-module, the first inclusion of \eqref{eq :153} is an 
equality.

The claim for $V\in\Rep_K(\mathbb{Z})$ follows 
similarly.
\end{proof}

We now define regular singular modules.

\begin{definition}[Regular singular differential modules]\label{Def : reg sing diff mod}
Let $A$ be a differential $K$-algebra 
without exponents nor logarithm. 
If $M$ is a differential module over $A$ we say that $M$ 
is \emph{regular singular} if 
\begin{enumerate}
\item $M$ is finite free as $A$-module; 
\item $M$ is trivialized by $E_A$ (as a differential module).
\end{enumerate}
A morphism of regular singular modules is just a morphism of 
differential modules over $A$.
We denote by 
\begin{equation}
\Reg(A)
\end{equation} the full sub-category of 
$\partial-\Mod(A)$ formed by regular singular modules. 
As in Notation \ref{Notation : mon_rho}, when necessary 
we may explicit the dependence on the choices
\begin{equation}\label{eq: Reg_rho}
\Reg(A)=
\Reg_{\rho}(A)=
\Reg_{\rho,\partial}(A)=
\Reg_{\rho,\partial,\gamma}(A)\;.
\end{equation}
\end{definition}
\begin{lemma}\label{Lemma : ext of reg is reg}
Let $A$ be a differential $K$-algebra 
without exponents nor logarithm and let $M,N\in\Reg(A)$. 
Then, the differential module of internal hom 
$\Hom_A(M,N)$ and internal tens $M\otimes_A N$ 
are regular singular. 
In particular, the dual of a regular singular module is regular singular
\end{lemma}
\begin{proof}
By assumptions we have $E_A$-linear isomorphisms
$M\otimes_AE_A\simto E_A^m$ and $N\otimes_AE_A\simto E_A^n$ 
commuting with the connections, where $E_A=\bs{1}$ denotes the unit 
object. Since $M$ is finite free over $A$, by \cite[\S5, N.3, Proposition 7]{Bou-Alg-II} we have isomorphisms 
$\Hom_A(M,N)\otimes_AE_A
\cong
\Hom_{E_A}(M\otimes_AE_A,N\otimes_AE_A)
\cong
\Hom_{E_A}(E_A^m,E_A^n)
\cong
\Hom_{E_A}(E_A,E_A)^{nm}
\cong E^{nm}_A$. 
It is not hard to show that these isomorphisms commute with the 
connections, therefore $\Hom_A(M,N)$ is trivialized by $E_A$, 
i.e. it is regular singular.

A similar argument shows that 
$
M\otimes_AN\otimes_A E_A
\cong
(M\otimes_AE_A)\otimes_{E_A}(N\otimes_AE_A)
\cong
E_A^m\otimes_{E_A}E_A^n
\cong 
E_A^{mn}
$.
\end{proof}

\begin{proposition}\label{Prop : Reg stable by extensions}
Let $A$ be a differential $K$-algebra 
without exponents nor logarithm. 
Assume moreover that $\partial:E_A\to E_A$ is surjective. 
Then the category of regular singular differential modules is stable by 
Yoneda extensions. This is the case in particular for $A=K[t,t^{-1}]$ (cf. Proposition \ref{Prop : Ktt-1 partial epi}).
\end{proposition}
\begin{proof}
Let $0\to M\to Q\to N \to 0$ be an exact sequence where 
$M,N\in \Reg(A)$. 
Since $N$ is free as $A$-module, it is flat and the sequence 
$0\to M\otimes_AE_A\to Q\otimes_AE_A\to N\otimes_AE_A\to 0$ is exact. 
Therefore, $Q\otimes_AE_A$ is an extension of $E_A^n$ by $E_A^m$ in 
the category of differential modules over $E_A$. By \eqref{eq : yoneda ext} 
we have $Ext(E_A^n,E_A^m)=H^1(E_A^{nm})=H^1(E_A)^{nm}$. Now, 
since the derivation $\partial$ is surjective on $E_A$ we have $H^1(E_A)=0$ (cf. \eqref{eq : yoneda ext}). 
It follows that $Q\otimes_AE_A$ is necessarily isomorphic, 
as differential module over $E_A$, to the direct sum of 
$M\otimes_AE_A$ and $N\otimes_AE_A$.  The claim follows.
\end{proof}

\begin{theorem}\label{Thm. Equivalence}
Let $A$ be a differential $K$-algebra 
without exponents nor logarithm.
The restriction of the functor $\Mon$ to the category 
$\Reg(A)$ gives an equivalence of categories
\begin{equation}
\Mon\;:\;\Reg(A)\;\xrightarrow[]{\;\;\cong\;\;}\;
\Rep_K^{\mathrm{fin}}(\mathbb{Z})\;
\end{equation}
with quasi inverse $\Rm$.
Both $\Mon$ and $\Rm$ are exact functors on these 
categories and they preserve the dimension of the objects as well as
the operations of internal $\Hom$, 
$\otimes$, and duality. 
\end{theorem}
\begin{proof}
Let $M\in\Reg(A)$. By definition, this means that
$M\otimes_AE_A$ has a trivial action of $\nabla$. 
By definition, also the action of $[1]$ is trivial, because it is given by 
$1\otimes \sigma$. However, in general there is no basis of $M\otimes_{A}E_A$ trivializing simultaneously $\nabla$ and $[1]$. 
The fact that both the actions are trivial, implies that we have canonical 
isomorphisms (cf. \eqref{eq : trivial downup} and \eqref{eq: trud})
$(M\otimes_AE_A)^{\nabla=1}\otimes_KE_A\simto M\otimes_AE_A$
and 
$(M\otimes_AE_A)^{1\otimes\sigma=1}
\otimes_AE_A\simto M\otimes_AE_A$ (cf. Lemma \ref{Lemma: up down})
commuting with the connections and the actions of the monodromy 
operators.

In particular, we have an $E_A$-linear isomorphism 
\begin{eqnarray}\label{eq : Mon(M)oE=MoE}
\Mon(M)\otimes_KE_A&\cong&M\otimes_AE_A\;.
\end{eqnarray}
This shows that $M$ and $\Mon(M)$ share the same dimension.
By Lemma \ref{Lemma: up down} it follows that
\begin{equation}
\Rm(\Mon(M))\;=\;(\Mon(M)\otimes_KE_A)^{[1]=1}\;
\cong\;(M\otimes_AE_A)^{1\otimes\sigma=1}=M\;.
\end{equation}
Let now $V\in\Rep_K^{\mathrm{fin}}(\mathbb{Z})$. 
By Proposition \ref{Lemma : every V is trivialized}, 
the action of the monodromy on $V\otimes_KE_A$ is trivial. 
On the other hand, by definition $V\otimes_KE_A$ 
has a trivial connection $1\otimes\partial$
and hence as before we have isomorphisms
$(V\otimes_KE_A)^{1\otimes\partial=0}
\otimes_KE_A\simto V\otimes_KE_A$ and 
$(V\otimes_KE_A)^{[1]=1}
\otimes_KE_A\simto V\otimes_KE_A$ 
(cf. Lemma \ref{Lemma: up down}).
 We then have an isomorphism commuting with the actions of the 
connections and the monodromies.
\begin{eqnarray}
\Rm(V)\otimes_AE_A&\cong&V\otimes_KE_A\;.
\end{eqnarray}
It follows that $V$ has the same dimension as $\Rm(V)$ and that
\begin{equation}
\Mon(\Rm(V))\;=\;(\Rm(V)\otimes_AE_A)^{\nabla=0}\;
\cong\;
(V\otimes_KE_A)^{1\otimes\partial=0}=V\;.
\end{equation}

Let us now check the fully faithfulness of the functors.
If $M,N\in\Reg(A)$, we know that
$\Hom_A(M,N)$ is regular singular. Therefore, by Lemma 
\ref{Lemma: Mon and R commutes to solutions} we have
\begin{equation}
\Hom_A^\nabla(M,N)\;=\;\Hom_A(M,N)^{\nabla=0}\;=\;
\Mon(\Hom_A(M,N))^{\mathbb{Z}}\;.
\end{equation}
It is then enough to show that we have an isomorphism 
\begin{equation}
\Mon(\Hom_A(M,N))\cong
\Hom_K(\Mon(M),\Mon(N))
\end{equation}
of $K$-linear representations of $\mathbb{Z}$. 
In other words, we want to prove that the morphism 
\eqref{eq:arrowonhom} is an isomorphism. 
This follows again from 
\eqref{eq : Mon(M)oE=MoE} because \eqref{eq:arrowonhom} is obtained as composition of the following morphisms
\begin{eqnarray}
\Hom_A(M,N)\otimes_AE_A&\;\cong\;&
\Hom_{E_A}(M\otimes_AE_A,N\otimes_AE_A)\nonumber\\
&\;\cong\;&
\Hom_{E_A}(\Mon(M)\otimes_KE_A,
\Mon(N)\otimes_KE_A)\\
&\;\cong\;&
\Hom_{K}(\Mon(M),\Mon(N))\;.
\end{eqnarray}
Where the first isomorphism follows from the fact that $M$ is finite free as 
$A$ module \cite[\S5, N.3, Proposition 7]{Bou-Alg-II}; the second from 
\eqref{eq : Mon(M)oE=MoE};  and the third one again from 
\cite[\S5, N.3, Proposition 7]{Bou-Alg-II} and the fact that $\Mon(M)$ 
is finite dimensional.

The rest of the proof is routine. 
\end{proof}

\begin{corollary}\label{Cor : descent to ktt-1}
Let $A$ be a differential $K$-algebra 
without exponents nor logarithm and $M$ a
differential module over $A$ which is finite free as 
$A$-module. The following are equivalent:
\begin{enumerate}
\item\label{Cor : descent to ktt-1-1} $M$ is regular singular;
\item\label{Cor : descent to ktt-1-2} 
$M$ is finite free over $A$ and it has a basis in 
which the matrix of the connection has coefficients in 
$K$ (cf. Example \ref{Example : Jordan}).
\end{enumerate}
In particular, the scalar extension functor 
\begin{equation}\label{eq : sc descent to Gm}
-\otimes_{K[t,t^{-1}]}A\;:\;
\Reg(K[t,t^{-1}])\;
\xrightarrow{\;\;\sim\;\;}\;\Reg(A)
\end{equation}
is an equivalence of categories.
\end{corollary}
\begin{proof}
The equivalence \eqref{eq : sc descent to Gm} follows 
from the fact that $-\otimes_{K[t,t^{-1}]}A$ commutes 
with the functors $\Mon$ on $\Reg(A)$ and 
$\Reg(K[t,t^{-1}])$ and the fact that they induces equivalence with the 
same category $\Rep^{\textrm{fin}}_K(\mathbb{Z})$.

Therefore, in order to prove that 
\eqref{Cor : descent to ktt-1-1} and
\eqref{Cor : descent to ktt-1-2} are equivalent, 
it is not restrictive to assume $A=K[t,t^{-1}]$.

If a differential module $M$ over $K[t,t^{-1}]$ has a basis in which the 
matrix of the connection has coefficients in $K$, 
then by Example \ref{Example : Jordan} $M$ 
is extension of modules of type $N(a)$ and it is hence regular singular by 
Proposition \ref{Prop : Reg stable by extensions}.  

Let us now assume that $M\in\Reg(K[t,t^{-1}])$. By Jordan classification of 
endomorphisms (cf. \eqref{eq : J(lambda,n)}), $\Mon(M)$ is extension of 
rank one objects of the form $V_\lambda$ (cf. \eqref{eq : Vlambda}), and 
by Theorem \ref{Thm. Equivalence} it follows that $M$ is extension of rank 
one differential modules of the form $N(a)$ (cf. Example 
\ref{Exemple: Vlambda N(a)}).
Yoneda extensions of $N(a)$ by $N(b)$ are controlled by $H^1(N(a-b))$ 
(cf. \eqref{eq : yoneda ext}) and a straightforward 
computation proves that $M$ always have a basis in which 
the matrix of the connection has coefficients in $K$ (and 
it is in Jordan form). 
\if{
Now, objects in $\Rep_K^{\mathrm{fin}}(\mathbb{Z})$ 
have a Jordan-Holder sequence formed by 
one dimensional sub-quotients (cf. section 
\ref{Section: matrix of 1}). 
For $\lambda\in K^\times$ let us denote by 
$V_\lambda$ the one dimensional 
$\sigma$-module over $K$ where $\sigma$ is just 
the multiplication by the scalar $\lambda$. Every one 
dimensional $\sigma$-modules over $K$ has this form.
The rule \eqref{eq : MonNa=Vlambda} provides the value 
of the functor on rank one objects:
\begin{equation}
\Rm(V_\lambda)\;=\;N(a)\;,
\qquad\gamma(a)=\lambda^{-1}\;.
\end{equation}
On the other hand, the Yoneda extension groups are (cf. Sections \ref{} and \ref{})
\begin{eqnarray}
Ext(V_\lambda,
\V_\mu)
&\;=\;&
\left\{
\sm{0&\textrm{ if }&\mu\neq\lambda\\
K&\textrm{ if }&\mu=\lambda\\}\right.\\
Ext(N(a),N(b))
&\;=\;&
\left\{
\sm{0&\textrm{ if }&a\neq b\\
K&\textrm{ if }&a=b\\}\right.
\end{eqnarray}
The exactness of the functors $\Rm$ and $\Mon$ imply 
that non trivial extensions should correspond by these 
functors. Moreover, \emph{over the ring $K[t,t^{-1}]$}, 
every differential module having a Jordan-Holder 
sequence made by rank one sub-quotients all 
isomorphic to modules of type $N(a)$ is regular singular 
(cf. Proposition \ref{Prop : Reg stable by extensions}).\footnote{This fails for general ring $A$ (cf. Remark \ref{Rk : extensions of regular, possibly not regular}).}
Now, Yoneda extensions of $N(a)$ by itself always have 
a basis in which the matrix of the connection has 
coefficients in $K$. 
Therefore, regular singular modules always have a basis in which 
the matrix of the connection has coefficients in $K$ (and 
it is in Jordan form). 
}\fi
\end{proof}
\begin{remark}\label{Remark : matrix Jordan}
Thanks to Corollary \ref{Cor : descent to ktt-1}, 
we can compute explicitly the functors as follows. 
Let $V\in\Rep_K^{\mathrm{fin}}(\mathbb{Z})$ be a 
$K$-linear representation $\mathbb{Z}$ such that the 
monodromy operator $[1]$ has matrix $J$ in Jordan 
canonical square blocks of the form $J(\lambda,n)$ 
(cf. \eqref{eq : J(lambda,n)}). 
Then $M:=\Rm(V)$ has a basis in which the connection 
has matrix $G\in M_n(K)$ that can be obtained from $J$ 
by replacing every Jordan block $J(\lambda,n)$ 
with the block $J(a,n)$, where $a=\gamma^{-1}
(\lambda^{-1})\in\widetilde{K/\mathbb{Z}}$. 

Reciprocally, if $M\in\Reg(A)$ has a basis in which the matrix $G$ 
of its connection has constant coefficients, 
then we may change the basis of $M$ by a 
base change whose matrix has coefficients in $K$ in 
order to turn $G$ into its Jordan normal form with blocks $J(a,n)$ (cf. 
Example \ref{Example : Jordan}). 
Then $V:=\Mon(M)$ has a matrix of 
$\sigma$ which is obtained by the Jordan form of $G$ 
replacing every block $J(a,n)$ with $J(\lambda,n)$.
\end{remark}
There may exists a base change of $M$, by a matrix 
which is \emph{not} in $GL_n(K)$, such that the 
matrix of the connection is constant both before and 
after the base change. In this case, Remark 
\ref{Remark : matrix Jordan} implies that the Jordan 
form of the matrix of the connection is the same (up to 
permutation of the Jordan blocks) 
before and after the base change. We then obtain the following Corollary.
\begin{corollary}
\label{Cor : independence of eigenvalues}
Let  $A$ be a differential $K$-algebra without exponents 
nor logarithms. 
Let $M\in \Reg(A)$ be a $n$-dimensional regular singular module. 
If $G$ is a matrix of the connection of $M$ with 
coefficients in $K$, and if $\{a_1,\ldots,a_n\}$ is the 
multiset of the eigenvalues of $G$ (i.e. eigenvalues with 
multiplicity), then the image of $\{a_1,\ldots,a_n\}$ in 
$K/\mathbb{Z}$ is a multiset which is 
invariant by base changes of  
$M$ (equivalently, it is invariant by isomorphisms).
\hfill$\Box$
\end{corollary}
\begin{definition}[Exponents]
Let  $A$ be a differential $K$-algebra without exponents 
nor logarithms and $M\in\Reg(A)$ an $n$-dimensional regular singular module. With the notation as 
in Corollary \ref{Cor : independence of eigenvalues}, 
we call 
\begin{equation}
\Exp(M)\;=\;\{a_1,\ldots,a_n\}\subset K/\mathbb{Z}
\end{equation}
the multiset of the \emph{exponents of $M$}.
We call $K/\mathbb{Z}$ the \emph{group of 
exponents.}

As in Notation \ref{Notation : mon_rho}, when necessary 
we also use the notation
\begin{equation}\label{eq: Reg_rho}
\Exp(M)=
\Exp_{\rho}(M)=
\Exp_{\rho,\partial}(M)=
\Exp_{\rho,\partial,\gamma}(M)\;.
\end{equation}
\end{definition}
\begin{remark}
We point out the dependence on the choice of 
$\rho:K[t,t^{-1}]\to A$. 
For instance, when $A=K[t,t^{-1}]$, 
we may consider the map $\rho'\;:\;K[t,t^{-1}]\to 
K[t,t^{-1}]$ sending $t$ onto $t^{-1}$. 
This map is an isomorphism of rings, but not commuting with 
$\partial_t:=t\frac{d}{dt}$ on both sides. Instead we 
have
\begin{equation}
\partial_{t^{-1}}\circ\rho' = \rho'\circ\partial_{t}
\end{equation}
and 
\begin{equation}
\partial_{t^{-1}}=-\partial_t\;.
\end{equation}
If $M$ is a differential module over $(A,\partial_t)$ 
with matrix $G\in M_n(A)$ 
(cf. \eqref{eq : matrix of conn}), then 
we can consider it as an $(A,\partial_{t^{-1}})$ as per 
Proposition \ref{Prop: change of deriv} 
and its matrix becomes $-G$. 
It follows that $(M,\nabla)$ is regular over 
$(K[t,t^{-1}],\partial_t)$ if and only if so is
$(M,-\nabla)$ over 
$(K[t,t^{-1}],\partial_{t^{-1}})$. However
the exponents changes sign
\begin{equation}
\Exp_\rho((M,\nabla))\;=\;-
\Exp_{\rho'}((M,-\nabla))\;.
\end{equation}
A similar phenomena arises in $p$-adic framework when $A$ is the ring of 
analytic functions over an open annulus. In this case, a change of variable 
drastically transforms the matrix of the connection and the isomorphism 
classes of regular singular modules. In particular, the category of regular 
singular modules is not stable under change of coordinates. 
This makes improbable that a criterion involving the norms or the 
valuations of the coefficients my detect the fact 
that the module is regular singular. 
\end{remark}
\begin{remark}\label{Rk : extensions of regular, possibly not regular}
Let $A$ be a differential $K$-algebra 
without exponents nor logarithm.
We point out that $A$ 
is allowed to have non trivial ideals stable 
by the derivation. 
But these ideals are differential modules 
that are not regular singular. In other words, the unit object of 
$\partial-\Mod(A)$ may have non trivial sub-objects, but 
it does not have sub-objects as an object of $\Reg(A)$. 

On the other hand there are no restrictions on the dimension of
\begin{eqnarray*}
\mathrm{coker}(\partial:A\to A)\;=\;
\H^1(A)\;=\;
H^1(N(a)^*\otimes N(a))\;=\;
\Ext(N(a),N(a))\;.
\end{eqnarray*}
Therefore, we may have lot of Yoneda extensions of 
$N(a)$ by itself, but only a one dimensional sub-space 
of $\Ext(N(a),N(a))$ is formed by regular singular ones. 
In other words, we may have in $\partial-\Mod(A)$ some 
exact sequence $0\to N(a)\to M\to N(a)\to 0$ where $M$ 
is not regular singular. 

A similar phenomena shows up in the category of differential modules 
over a $p$-adic annulus. In that case it is known that, when $a-b$ is a 
$p$-adic non Liouville number, $\Ext(N(a),N(b))\cong H^1(a-b)$ 
is infinite dimensional (cf. \cite[5.5]{Ro-IV}). 
This tells us that some if these extensions are not trivialized by $E_A$ and 
hence have some solution which is not polynomial in $t^a$ and $\log(t)$.
\end{remark}

\begin{remark}
When $A=K\f{t}$, it is possible to prove that 
Katz's canonical extension functor $\mathrm{Can}:
\partial-\Mod(K\f{t})\xrightarrow{\quad}\partial-\Mod(K[t,t^{-1}])$ 
(cf. \cite{Katz-Can})  
is compatible with the functors 
$\Rm$ and $\Mon$.
\end{remark}

\if{
\section{Regular singular modules with prescribed exponents}
Let $A$ be a $K$ algebra without exponents nor logarithm.
Let us denote by $\mathcal{G}_{\textrm{reg}}$ the Tannakian group of the 
category $\Reg(A)$. By the equivalence of Theorem \ref{Thm. Equivalence}, 
$\Reg(A)$ and $\mathrm{Rep}_K^{\mathrm{fin}}(\mathbb{Z})$ share the 
same Tannakian group. Therefore $\mathcal{G}_{\textrm{reg}}$ is 
isomorphic to the algebraic envelop of $\mathbb{Z}$, which is 
known to be
\begin{equation}
\mathcal{G}_{\textrm{reg}}\;=\;
\mathcal{D}(K^\times)\times \mathbb{G}_a\;,
\end{equation}
$\mathcal{D}(K^\times)$ is the diagonal group associated with 
the group $K^\times$. 
Let $H$ be a subgroup of $K/\mathbb{Z}$, $\widetilde{H}\subset\widetilde{K/\mathbb{Z}}$ its lifting, and  
$\gamma(H)$ its image in $K^\times$ by the 
isomorphisms \eqref{eq : isom gamma def}.
We denote by 
\begin{equation}
\Reg(A;H)
\end{equation}
the full subcategory of 
$\Reg(A)$ formed by regular singular 
differential modules over $A$ 
whose exponents lie in $H$. Then $\Reg(A;H)$ is a 
Tannakian category whose Tannakian group 
$\mathcal{G}_{\textrm{reg},H}$ is 
\begin{equation}
\mathcal{G}_{\textrm{reg},H}\;:=\;
\mathcal{D}(\gamma(H))\times \mathbb{G}_a\;.
\end{equation}
The affine algebra of $\mathcal{D}(H)$ 
being the group algebra $K\lr{H}$ of $H$, the affine algebra 
of  $\mathcal{G}_{\textrm{reg},H}$ 
is then (with an evident meaning of the symbols)
\begin{equation}
E_{A,H}\;:=\;
A[t^{\widetilde{H}},\ell]\;.
\end{equation}
The functors $\Mon$ and $\Rm$ establish an equivalence 
of categories between $\Reg(A;H)$ and 
\begin{equation}
\Rep_{K}^{\textrm{fin}}(\mathbb{Z},\gamma(H))\;,
\end{equation}
where the latter is the category of finite dimensional 
$K$-linear representations of $\mathbb{Z}$ such that 
the eigenvalues of the monodromy operator $[1]$ belong 
to $\gamma(H)\subset K^\times$.

\begin{remark}
The full subcategories 
$\Reg(A;H)^{ss}$ and $\Rep_K^{\mathrm{fin}}(\mathbb{Z},\gamma(H))^{ss}$ formed by semisimple objects are Tannakian with group 
$\mathcal{D}(\gamma(H))$.
\end{remark}

In particular, let $H=\mathbb{Q}/\mathbb{Z}$ be the 
torsion subgroup of $K/\mathbb{Z}$. Then 
$\gamma(\mathbb{Q}/\mathbb{Z})=\mu_\infty\subset 
K^\times$ is the sub-group of all roots of unity. 

One sees that $\Rep_{K}^{\textrm{fin}}(\mathbb{Z},
\mu_\infty)$ is precisely the category of 
finite dimensional 
$K$-linear representations $V$ 
of $\mathbb{Z}$ such that the action of $\mathbb{Z}$ 
on the semi-semplified $V^{ss}$ 
factors through a finite quotient. In other words 
the map $\rho\;:\;
\mathbb{Z}\to\mathrm{Aut}_K(V^{ss})$ 
(cf. \eqref{eq : rho def representation}) has finite 
image. Therefore, one has 
\begin{equation}
\mathcal{G}_{\mathrm{reg},
\mathbb{Q}/\mathbb{Z}}\;=\;
\widehat{\mathbb{Z}}\times\mathbb{G}_a\;,
\end{equation}
where $\widehat{\mathbb{Z}}$ is the profinite 
completion of $\mathbb{Z}$.

\section{Decomposition by the maximal regular singular sub-module}
If $A=K\f{t}$, every differential module is direct sum of a 
regular one and a completely irregular one. 
we have a decomposition
\begin{proposition}
\end{proposition}
}\fi

\if{
\section{Decomposition by the maximal regular singular sub-module}

We can define the maximal regular singular sub-module as the Reg of its 
Monodromy representation.

Is it a direct summand ?

Is the quotient free of regularity ?

This has to do with exactness of REG and monodromy functors.

Moreover, we may have more extensions than expected ...
}\fi
\begin{small}
\bibliographystyle{amsalpha}
\bibliography{bib}
\end{small}

\end{document}